\documentclass[12pt,reqno]{amsart}
\usepackage{amsaddr}
\usepackage{amssymb,latexsym,amsmath,epsfig,amsthm,mathrsfs}
\usepackage{rotating}
\usepackage{graphicx}
\usepackage{amssymb}
\usepackage{lineno}
\usepackage{enumitem}
\usepackage{cite}
\usepackage[top=3cm, bottom=3cm, left=3cm, right=3cm]{geometry}
\usepackage[usenames]{color}
\usepackage[colorlinks=true,
linkcolor=blue,
filecolor=blue,
citecolor=blue]{hyperref}

\newtheorem{theorem}{Theorem}[section]
\newtheorem{corollary}[theorem]{Corollary}
\newtheorem{lemma}[theorem]{Lemma}
\theoremstyle{definition}
\newtheorem{definition}[theorem]{Definition}
\theoremstyle{remark}

\theoremstyle{definition}

\newtheorem*{example}{Example}
\numberwithin{equation}{section}

\begin{document}
	
	\title[Signed sumsets and restricted signed sumsets in groups and fields]{Signed sumsets and restricted signed sumsets in groups and fields}
	
	
\author[R K Mistri]{Raj Kumar Mistri}
\address{\em{\small Department of Mathematics, Indian Institute of Technology Bhilai\\
		Durg, Chhattisgarh, India\\
		e-mail: rkmistri@iitbhilai.ac.in}}


\author[N Prajapati]{Nitesh Prajapati$^{*}$}
\address{\em{\small Department of Mathematics, Indian Institute of Technology Bhilai\\
		Durg, Chhattisgarh, India\\
		email: niteshp@iitbhilai.ac.in}}

\thanks{$^{*}$Corresponding author}
\thanks{$^{*}$The research of the author is supported by the UGC Fellowship (NTA Ref. No.: 211610023414)}



\begin{abstract}
Let $A = \{a_1, \ldots, a_k\}$ be a nonempty finite subset of an additive abelian group $G$. For a nonnegative integer $h$, the \emph{$h$-fold signed sumset} of $A$, denoted by $h_{\pm} A$, is defined by
$$
h_{\pm} A = \Biggl\{\sum_{i = 1}^{k} \lambda_i a_i : \lambda_i \in \{-h, \ldots, h\}, \ \sum_{i = 1}^{k} |\lambda_i| = h \Biggr\},
$$
and the \emph{restricted $h$-fold signed sumset}, denoted by $h_{\pm}^\wedge A$, is defined by
$$
h_{\pm}^\wedge A = \Biggl\{\sum_{i = 1}^{k} \lambda_i a_i : \lambda_i \in \{-1, 0, 1\}, \ \sum_{i = 1}^{k} |\lambda_i| = h \Biggr\}.
$$
We study direct and inverse problems for these signed sumsets, namely determining extremal bounds for their sizes and characterizing the structure of sets $A$ attaining these bounds. While such problems have been extensively studied and resolved in the additive group of integers, comparatively little is known in general abelian groups, especially for restricted signed sumsets. In this paper, we investigate the signed sumset $h_{\pm} A$ in arbitrary (not necessarily finite) abelian groups under the condition $A \cap (-A) \neq \varnothing$. We further analyze both $h_{\pm} A$ and $h_{\pm}^\wedge A$ when $A \cap (-A)$ has a prescribed size. These results are extended to generalized signed sumsets $H_{\pm} A = \bigcup_{h \in H} h_{\pm} A$, where $H$ is a finite set of nonnegative integers, with particular attention to $[0,h]_{\pm} A$. Furthermore, using the polynomial method, we establish nontrivial lower bounds for $|h_{\pm}^\wedge A|$ in arbitrary fields. In addition, for $h = 2, 3, 4$, we derive lower bounds for $|h_{\pm} A|$ in arbitrary fields under the condition $A \cap (-A) = \varnothing$.
\end{abstract}

	\maketitle
\tableofcontents

\section{Introduction}
	Throughout this paper, let $G$ denote an abelian group, written additively. Let $\mathbb{Z}$ denote the set of integers, and let $\mathbb{F}$ denote an arbitrary field. For integers $u$ and $v$ such that $u \leq v$, we denote the set $\{n \in \mathbb{Z} : u \leq n \leq v\}$ by $[u, v]$. The cardinality of a finite set $A$ is denoted by $|A|$. Let $A = \{a_1, \ldots, a_k\}$ be a nonempty finite subset of $G$, and let $h$ be a nonnegative integer. Then the {\it $h$-fold sumset} $hA$ of the set $A$ is defined as
	\[hA = \left\{\sum_{i = 1}^{k}\lambda_i a_i: \lambda_i \in [0, h] ~\text{for}~ i = 1, \ldots, k ~\text{and}~ \sum_{i = 1}^{k}\lambda_i = h \right\}.\]
The {\em restricted $h$-fold sumset} $h^{\wedge}A$ of the set $A$ is defined as
	\[h^{\wedge}A = \left\{\sum_{i = 1}^{k}\lambda_i a_i: \lambda_i \in [0, 1] ~\text{for}~ i = 1,\ldots, k ~\text{and}~ \sum_{i = 1}^{k}\lambda_i = h \right\}.\]

There are two kinds of problems associated with the sumsets defined above, as well as with other types of sumsets introduced later in this paper: direct problems and inverse problems. The problem of estimating the optimal size of a sumset falls under the first category. A set for which the corresponding sumset attains this optimal size is called an \emph{extremal set}. The problem of characterizing such extremal sets belongs to the second category.

The study of sumsets in groups has a history of more than two centuries. Let $\mathbb{Z}_p$ denote the cyclic group of prime order $p$. A classical result in this area is the \emph{Cauchy--Davenport theorem} (see \cite{cauchy, dav1, dav2}), which asserts that
\[
|A + B| \geq \min(p, |A| + |B| - 1),
\]
where $A, B \subseteq \mathbb{Z}_p$ and $A + B = \{a + b : a \in A,\, b \in B\}$. Dias da Silva and Hamidoune \cite{dias} determined the optimal lower bound for the restricted $h$-fold sumset $h^{\wedge}A$ in an arbitrary field $\mathbb{F}$ using exterior algebra. This result was subsequently reproved by Alon, Nathanson, and Ruzsa \cite{alon1, alon2} via the polynomial method, which has since become a central and powerful technique in additive combinatorics, particularly for problems previously considered intractable. For further background and a comprehensive account of results on sumsets, we refer the reader to the books of Nathanson \cite{nath}, Tao and Vu \cite{tao}, Grynkiewicz \cite{gryn2013book}, and Bajnok \cite{bajnok2018}. 
	
Recently, two other types of sumsets have appeared in the work of Bajnok, Ruzsa, and other researchers (see {\cite{bajnok-matzke2015,bajnok-matzke2016,bajnok-ruzsa2003,bajnok2018,bhanja-kom-pandey2021,bhanja-pandey2019,klopsch-lev2003,klopsch-lev2009,mmp2024}}): the {\it $h$-fold signed sumset} $h_{\pm}A$ and the {\it restricted $h$-fold signed sumset} $h_{\pm}^\wedge A$ which are defined as follows:
	\[ h_{\pm}A = \Biggl\{\sum_{i = 1}^{k}\lambda_ia_i : \lambda_i \in [- h, h] ~\text{for each}~ i \in [1, k],\sum_{i = 1}^{k}|\lambda_i| = h\Biggl\},\] 
	and 
	\[ h_{\pm}^\wedge A = \Biggl\{\sum_{i = 1}^{k}\lambda_ia_i : \lambda_i \in [- 1, 1] ~\text{for each}~ i \in [1, k],\sum_{i = 1}^{k}|\lambda_i| = h\Biggl\}.\] 

Signed sumsets have been studied in various mathematical contexts. The concept of the $h$-fold signed sumset was initially introduced by Bajnok and Ruzsa \cite{bajnok-ruzsa2003}, who investigated it in relation to the independence number of subsets within an abelian group $G$ (see also \cite{bajnok2000, bajnok2004}). Later, Klopsch and Lev \cite{klopsch-lev2003, klopsch-lev2009} explored $h_{\pm}A$ in connection with the diameter of the group $G$ when measured with respect to the set $A$. Bajnok and Matzke initiated a detailed study of the $h$-fold signed sumsets. To state their result, we need the following definition from \cite{bajnok2018}. We also follow the notation used in \cite{bajnok2018}.
\begin{definition}\label{subsec-survey-signed-sumsets-def1}
Let $G$ be an additive abelian group, and let $m$ and $h$ be positive integers. Let $H$ be a nonempty finite set of nonnegative integers. We define the following:
\begin{enumerate}
	\item $\mathrm{Sym}(G, m) = \{A \subseteq G : A = - A, |A| = m\}$. 
	\item $\mathrm{Asym}(G, m) = \{A \subseteq G : A \cap (- A) = \emptyset, |A| = m\}$.
	\item $\mathrm{Nsym}(G, m) = \{A \subseteq G : A \not\in \mathrm{Sym}(G, m), A \setminus \{a\} \in \mathrm{Sym}(G, m - 1) ~\text{for some}~ a \in A\}$.
	\item $\mathcal{A}(G, m) = \mathrm{Asym}(G, m) \cup \mathrm{Sym}(G, m) \cup \mathrm{Nsym}(G, m)$.
		\item {\it $H$-fold sumset of $A$:} 
		\[HA = \bigcup_{h \in H}hA.\]
		\item {\it $H$-fold signed sumset of $A$:}
		\[H_{\pm}A = \bigcup_{h \in H}h_{\pm}A.\]
		\item {\it Restricted $H$-fold signed sumset of $A$:}
		\[H_{\pm}^{\wedge}A = \bigcup_{h \in H}H_{\pm}^{\wedge}A.\]
		\item For positive integer $m$ and $h$, we have
		\[\rho_{\pm}(G, m, H) = \min\{|H_{\pm}A| : A \subseteq G, |A| = m\},\]
		and
		\[\rho_{\pm}^\wedge (G, m, H) = \min \{|H_{\pm}^\wedge A| : A \subseteq G, |A| = m\}.\]
	\end{enumerate}
In all cases, we set $0A=0_{\pm}A=0^{\wedge}A=0_{\pm}^{\wedge}A=\{0\}$. If $H=\{h\}$, we simply write $h$ instead of $\{h\}$ in the notation.
\end{definition}

The direct and inverse problems for signed sumsets and restricted signed sumsets seem difficult in arbitrary abelian groups, even in the cyclic group of prime order. Bajnok and Matzke proved the following result.

\begin{theorem}[\cite{bajnok-matzke2015}, Theorem $3$]\label{bajnok-matzke-ss-min-group-thm}
	Let $h$ and $m$ be positive integers. Let $G$ be a finite abelian group. Then
	\[\rho_{\pm}(G, m, h) = \min \{|h_{\pm}A| : A \in \mathcal{A}(G, m)\}.\]	
\end{theorem}

In a subsequent work, they investigated the cases when $\rho_{\pm}(G, m, h) = \rho(G, m, h)$ for elementary abelian groups \cite{bajnok-matzke2016}. More recently, Bhanja and Pandey \cite{bhanja-pandey2019} studied the direct and inverse problems for $h$-fold signed sumsets in the additive group $\mathbb{Z}$ of integers. For the restricted signed sumset $h^{\wedge}_{\pm}A$, relatively little is known, even in the additive group of integers $\mathbb{Z}$. In this direction, Bhanja, Komatsu, and Pandey \cite{bhanja-kom-pandey2021} recently resolved  both the direct and inverse problems for $h^{\wedge}_{\pm}A$ in $\mathbb{Z}$ for the cases $h = 2$ and $h = |A|$. They proposed conjectures for the cases $3 \leq h \leq |A| - 1$, and confirmed the conjectures for $h = 3$. 
Mohan, Mistri, and Pandey \cite{mmp2024} made further progress on these conjectures and confirmed the conjectures for $h = 4$. Finally, these conjectures have been completely setteled by the authors in \cite{mn2026}. Recently,  Mohan \cite{mohan2026} has studied some extended inverse theorems for $h_{\pm}^\wedge A$ when $h \in \{2,3,k\}$.

Theorem \ref{bajnok-matzke-ss-min-group-thm} shows that it suffices to consider sets in the class $\mathcal{A}(G, m)$. A natural question then arises: what happens if we restrict our attention to sets from the class $\mathrm{Asym}(G, m)$ or from the class $\mathrm{Sym}(G, m)$? More generally, we investigate the $h$-fold signed sumset and the restricted $h$-fold signed sumset of a set $A$ when $A$ intersects exactly a fixed number of elements of $-A$. We extend this study to the more general signed sumset $H_{\pm}A$ for a nonempty finite subset $H$ of nonnegative integers.

For a nonnegative integer $s$ and a set $H$ of nonnegative integers, we define
\[
\rho_{\pm}^{(s)}(G, m, H) = \min \{|H_{\pm}A| : A \subseteq G, \ |A| = m, \ \text{and} \ |A \cap (-A)| = s\}.
\]
In particular, if $H = \{h\}$, then we denote $\rho_{\pm}^{(s)} (G, m, H)$ by $\rho_{\pm}^{(s)} (G, m, h)$. In this paper, we establish the following main results. 

\begin{theorem}\label{ss-group-ext-bajnok-thm:1}
	Let $m$ be positive integer, and let $H$ be nonempty finite set of nonnegative integers, and let $G$ be an additive abelian group. Then
	\[\rho_{\pm}(G, m, H) = \min \{|H_{\pm}A| : A \in \mathcal{A}(G, m)\}.\]
\end{theorem}

Above theorem generalizes Theorem \ref{bajnok-matzke-ss-min-group-thm}. Another result proved for the sumset $H_{\pm}A$ is Theorem \ref{ss-asym-group-thm:13}.  The following theorem establishes a nontrivial lower bound for $\rho_{\pm}^{(s)} (G, m, h)$. 

\begin{theorem}\label{ss-nasym-group-thm:5}
	Let $h$, $m$ and $s$ be positive integers with $s \leq m$. Let $G$ be an additive abelian group. Then 
	 \begin{equation}\label{ss-nasym-group-thm:5-eq:1}
			\rho_{\pm}^{(s)} (G, m, h) \geq \min (p(G), 2hm - hs - h + 1).
		\end{equation}
	The equality holds in \eqref{ss-nasym-group-thm:5-eq:1} if and only if $m \leq \dfrac{ s + p(G)}{2}$.
\end{theorem}
We have also proved some other results for $\rho_{\pm}^{(s)} (G, m, h)$ (see Theorem \ref{ss-group-thm:20}), and for $\rho_{\pm}(G, m, [0, h])$ (see Theorem \ref{ss-group-thm:2}). Furthermore, using the polynomial method, we obtain sharp lower bounds on the size of the signed sumsets $h_{\pm}A$ for $h \in \{2, 3, 4\}$ in an arbitrary field $\mathbb{F}$ (see Theorem \ref{ss-asym-prime-thm:9}, Theorem \ref{ss-asym-prime-thm:10} and Theorem \ref{ss-asym-prime-thm:11}). We also derive nontrivial lower bound for the size of the restricted $h$-fold signed sumset $h^{\wedge}_{\pm}A$ in $\mathbb{F}$. More precisely, we prove the following theorem.

\begin{theorem}\label{rss-field-thm:3}
	Let $\mathbb{F}$ be a field. Let $k$ and $h$ be positive integers such that $2 \leq h \leq k$ and $h - 1 \leq p(\mathbb{F})$. Let $A$ be a nonempty subset $\mathbb{F}$ with $|A| = k$. Let 
\[\theta = 2hk - \frac{h(3h - 1)}{2} - h |A \cap (- A)|.\]

Then
	\begin{equation*}
		|h_{\pm}^{\wedge} A|\geq
		\begin{cases}
			\min (p(\mathbb{F}), hk - h^2 + 1), & ~\text{if}~ \min(p(\mathbb{F}), \theta + 1) \leq hk - h^2 + 1;\\
			\theta + 1, & ~\text{if}~  hk - h^2 + 1 < \theta + 1 \leq p(\mathbb{F});\\
			\max (hk - h^2 + 1, \theta - \ell h + 1), & ~\text{if}~ hk - h^2 + 1  < p(\mathbb{F}) < \theta + 1,
		\end{cases}
	\end{equation*}
where $\ell$ is the least positive integer such that 
\[\theta - \ell h + 1 \leq p(\mathbb{F}) < \theta - (\ell - 1)h + 1.\]	
\end{theorem}


\subsection*{Organization of the paper}
The paper is organized as follows. In Section \ref{sec-signed-sum-group-aux-lemma}, we recall known results and establish several auxiliary lemmas needed in the sequel. In Section \ref{sec-signed-sum-group}, we extend a result of Bajnok and Matzke \cite{bajnok-matzke2015} (see Theorem \ref{ss-group-ext-bajnok-thm:1}) to arbitrary abelian groups and obtain a nontrivial bound for the size of the $H$-fold signed sumset. We also derive direct and inverse results for $H_{\pm}A$ under the assumption $A \cap (-A) \neq \varnothing$ in the case $H = \{h\}$. In addition, we study the functions $\rho_{\pm}^{(s)}(G,m,h)$ and $\rho_{\pm}^{(s)}(G,m,[0,h])$ for an arbitrary abelian group $G$. Using the polynomial method, we further obtain sharp lower bounds for $|h_{\pm}A|$ when $h \in \{2,3,4\}$ in an arbitrary field $\mathbb{F}$ under the condition $A \cap (-A) = \varnothing$. In Section \ref{sec-res-signed-sum-group}, we establish a nontrivial lower bound for the restricted signed sumset $h_{\pm}^{\wedge}A$ in $\mathbb{F}$. Moreover, we prove lower bounds for $|[0,h]_{\pm}^{\wedge}A|$ in a finite abelian group $G$ and establish the corresponding result to arbitrary fields. 
	
\section{Auxiliary results}\label{sec-signed-sum-group-aux-lemma}
In this section, we present some known results and also prove some auxiliary lemmas that are required for the proof of the main results. A set $A \subseteq G$ is called a \emph{$k$-term arithmetic progression with common difference $d$} (or simply an \emph{arithmetic progression with common difference $d$}) if $A = \{ a + i d : i \in [0, k-1] \}$ for some $a \in G$ and $0 \neq d \in G$. For an element $g$ of $G$, the subgroup generated by $g$ is denoted by $\langle g\rangle$. For a given subset $H \subseteq G$ and a nonzero integer$d$, we define $d \ast H = H \ast d = \{dh : h \in H\}$. In particular, $- H = - 1 \ast H$. Given a set of integers $A$ and an element $g$ of $G$, we define $A \ast g$ as the set $\{ag : a \in A\}$. The following theorem gives the optimal bound on the size of the sumset $hA$, and characterize the extremal sets.

\begin{theorem}[\cite{nath}, Theorem $1.3$ and Theorem $1.6$]\label{hfold-direct-thm}
	Let $h$ and $k$ be positive integers such that $h \geq 2$. Let $A$ be a set of integers with $|A| = k$. Then
	\begin{equation}\label{hfold-direct-thm-eq1}
		|hA| \geq hk - h + 1
	\end{equation}
	This lower bound is best possible. Moreover, the equality holds in \eqref{hfold-direct-thm-eq1} if and only if $A$ is a $k$-term arithmetic progression, provided $k \geq 2$.
\end{theorem}

DeVos \cite{devos2016} proved the following result in arbitrary groups (not necessarily finite).

\begin{theorem}\label{devos-extension}
	Let $A$ and $B$ be nonempty finite subsets of an additive group $G$. Then 
	\[|A + B| \geq \min (p(G), |A| + |B| - 1),\]
	where $A + B = \{a + b : a \in A, b \in B\}$. 
\end{theorem}

For a group $G$, we define $p(G)$ to denote the order of the smallest nontrivial subgroup of $G$, or $\infty$ if no such subgroup exists. For a positive integers $m$ and $h$, we define
\[\rho (G, m, h) = \min \{|hA| : A \subseteq G, |A| = m\}.\] 

Plagne proved the following result.

\begin{theorem}[\cite{plagne2006}]\label{plagne-group-h-fold-thm:1}
	Let $G$ be any finite abelian group of order $n \geq m$, and let $p(G)$ be the smallest prime divisor of $n$. Then 
	\[\rho (G, m, h) \geq \min (p(G),hm - h + 1),\] 
	with equality if, and only if $m \leq p(G)$.	
\end{theorem}

The authors proved the following result.
\begin{theorem}[\cite{mn-tf2026}, Lemma $5.5$ and Lemma $5.6$]\label{kemeperman-inv-abelian-ext-thm}
	Let $h \geq 2$ be an integer, and let $A_1, \ldots, A_h$ be nonempty subsets of $G$ with $|A_i| \geq 2$ for each $i \in [1, h]$. Let
	\[A_1 + \cdots + A_h = \{a_1 + \cdots + a_h : a_i \in A_i ~\text{for each}~ i \in [1, h]\},\]
	and let
	\begin{equation*}
		|A_1 + \cdots + A_h| <
		\begin{cases}
			p(G) - 1, & ~\text{if}~ h = 2;\\
			p(G), & ~\text{otherwise.}~
		\end{cases}
	\end{equation*}
	Then
	\[|A_1 + \cdots + A_h| = |A_1| + \cdots + |A_h| - h + 1,\]
	if and only if $A_1, \ldots, A_h$ are arithmetic progressions with the same common difference.
\end{theorem}

Bhanja and Pandey proved the following results for the signed sumset in the additive group of integers $\mathbb{Z}$.

\begin{theorem}[{\cite[Theorem $2.1$ and Theorem $2.2$]{bhanja-pandey2019}}] \label{bhanja-pandey-pos-inv1}
	Let $A$ be a finite set of $k \ge 3$ positive integers. Then
	\begin{equation}\label{bhanja-pandey-pos-inv1-eq1}
		|2_{\pm} A| \geq 4k - 2.
	\end{equation}
	Moreover, if equality holds in \eqref{bhanja-pandey-pos-inv1-eq1}, then $A = d \ast \{1, 3, \ldots, \ldots, 2k - 1\}$ for some positive integer $d$.
\end{theorem}

\begin{theorem}[{\cite[Theorem $2.3$]{bhanja-pandey2019}}]\label{bhanja-pandey-pos-dir}
	Let $h \ge 3$ be a positive integer, and let $A$ be a finite set of $k \geq 3$ positive integers. Then
	\[|h_{\pm} A| \ge 2hk - h + 1.\]
	This lower bound is best possible.
\end{theorem}

\begin{theorem}[\cite{bhanja-pandey2019}, Theorem $2.4$]\label{bhanja-pandey-pos-inv}
	Let $h \ge 3$ be a positive integer, and let $A$ be a finite set of $k \ge 3$ positive integers. If
	\[|h_{\pm} A| = 2hk - h + 1,\]
	then $A = d \ast \{1, 3, \ldots, \ldots, 2k - 1\}$ for some positive integer $d$.
\end{theorem}

\begin{lemma}\label{ss-nasym-group-lem:1}
	Let $h$ be a nonnegative integer. Let $A$ be a subset of $G$ such that $A \cap (- A) \neq \varnothing$. Then 
	\[h_{\pm}A = h(A \cup (-A)).\] 
\end{lemma}

\begin{proof}
	If $h = 0$ or $h = 1$, then the proof of the lemma is trivial. Now we assume that $h \geq 2$. It is easy to verify that $h_{\pm}A \subseteq h(A \cup (-A))$. Hence it is enough to show that $h(A \cup (-A)) \subseteq h_{\pm}A$. Let $x = a_1 + a_2 + \cdots + a_h \in h(A \cup (-A))$, where $a_i \in A \cup (-A)$ for $i = 1, \ldots, h$. Then $x = s_1a_{i_1} + s_2a_{i_2} + \cdots + s_ja_{i_j}$ for some $a_{i_1}, a_{i_2}, \ldots, a_{i_j} \in A \setminus \{0\}$ and $\sum_{i = 1}^j |s_i| \leq h$ with $a_{i_r}^2 \neq a_{i_s}^2$ for $r \neq  s$. If $\sum_{i = 1}^j |s_i| = h$, then $x \in h_{\pm}A$. If $\sum_{i = 1}^j |s_i| < h$, then we consider the following cases:
	
	\noindent {\textbf{Case 1}}($0 \in A$). In this case, we can write
	\[x = (h - \sum_{i = 1}^j |s_i|).0 + s_1a_{i_1} + s_2a_{i_2} + \cdots + s_ja_{i_j},\]
	and so $x \in h_{\pm}A$. 
	
	\noindent {\textbf{Case 2}}($0 \not\in A$).
	In this case, we can write  
	\[x = \frac{(h - \sum_{i = 1}^j |s_i|)}{2}a + \frac{(h - \sum_{i = 1}^j |s_i|)}{2} (- a) + s_1a_{i_1} + s_2a_{i_2} + \cdots + s_ja_{i_j},\] 
	where $a$ is an element of $A\cap (- A)$, and $(h - \sum_{i = 1}^j |s_i|)$ is an even number. It is easy to see that if $a, - a \not\in \{a_{i_1}, \ldots, a_{i_j}\}$, then $x \in h_{\pm}A$. Now assume that
	\[a_{i_m} = a ~\text{or}~ a_{i_n} = - a\]
	for some $m, n \in [1, j]$. Without loss of generality, we assume that $a = a_{i_m}$. We observe the following:
	\begin{enumerate}
		\item If $s_m > 0$, then
		\[x = \Biggl(\frac{(h - \sum_{i = 1}^j |s_i|)}{2} + s_m \Biggl) a + \frac{(h - \sum_{i = 1}^j |s_i|)}{2} (- a) + \sum_{\substack{r = 1 \\ r \neq m}}^j s_1a_{i_r} \in h_{\pm}A.\]
		\item If $s_m < 0$, then
		\[x = \frac{(h - \sum_{i = 1}^j |s_i|)}{2} a + \Biggl(\frac{(h - \sum_{i = 1}^j |s_i|)}{2} + (- s_m) \Biggl) (- a) + \sum_{\substack{r = 1 \\ r \neq m}}^j s_1a_{i_r} \in h_{\pm}A.\]
	\end{enumerate}
	Thus in each case, $x \in h_{\pm}A$. Since $x$ is an arbitrary element of $h(A \cup (- A))$, it follows that $h(A \cup (- A)) \subseteq h_{\pm}A$. This completes the proof.
\end{proof}

The following corollary of above theorem generalizes Theorem $2.6$ in \cite{bhanja-pandey2019} for arbitrary finite subsets of integers.

\begin{corollary}\label{ss-sumset-int-app:1}
	Let $h \geq 3$ and $m \geq 3$ be integers. Let $A$ be a subset of $\mathbb{Z}$ with $m$ elements. Then 
	\begin{equation}\label{ss-sumset-int-app:1-eq:1}
		|h_{\pm}A| \geq 2hm - h|A \cap (- A)| - h + 1.
	\end{equation}
	The lower bound in \eqref{ss-sumset-int-app:1-eq:1} is best possible.
\end{corollary}

\begin{proof}
	If $A \cap (- A) = \varnothing$, then the corollary follows from the fact that $h_{\pm}A = h_{\pm}A_{abs}$ and Theorem \ref{bhanja-pandey-pos-dir}. If $A \cap (- A) \neq \varnothing$, then it follows from Lemma \ref{ss-nasym-group-lem:1} that
	\[h_{\pm}A = h(A \cup (- A)),\]
	and so it follows from Theorem \ref{hfold-direct-thm} that
	\[|h_{\pm}A| = |h(A \cup (- A))| \geq h |A \cup (- A)| - h + 1 = 2hm - h|A \cap (- A)| - h + 1.\]   	  
\end{proof}

\begin{corollary}\label{ss-sumset-int-app:2}
	Let $h \geq 3$ and $m \geq 3$ be integers. Let $A$ be a subset of $\mathbb{Z}$ with $m$ elements. If 
	\begin{equation}\label{ss-sumset-int-app:2-eq1}
		|h_{\pm}A| = 2hm - h|A \cap (- A)| - h + 1,
	\end{equation}
	then following holds:
	\begin{enumerate}
		\item $|A \cap (- A)|$ is an even, then $A_{abs} = d \ast \{1, 3, \ldots, \ldots, 2m - 1\}$ for some positive integer $d$.
		\item $|A \cap (- A)|$ is an odd, then $A_{abs} = d \ast [0, m - 1]$ for some positive integer $d$.
	\end{enumerate}
\end{corollary}

\begin{proof}
	First, assume that $|A \cap (- A)| = 0$. In this case, $A \cap (- A) = \varnothing$. Since $h_{\pm}A = h_{\pm}A_{abs}$, it follows \eqref{ss-sumset-int-app:2-eq1} that
	\[|h_{\pm}A_{abs}| = 2hm - h|A \cap (- A)| - h + 1.\] 
	Hence it follows from Theorem \ref{bhanja-pandey-pos-inv} that $A_{abs} = d \ast \{1, 3, \ldots, \ldots, 2m - 1\}$ for some positive integer $d$. Now we assume that $|A \cap (- A)| \geq 1$. In this case, $A \cap (- A) \neq \varnothing$. Since $A \cap (- A) \neq \varnothing$, it follows that $0 \in A$ if and only if $|A \cap (- A)|$ is odd. It follows from \eqref{ss-sumset-int-app:2-eq1} and Lemma \ref{ss-nasym-group-lem:1} that
	\[|h_{\pm}A| = |h (A \cup (- A))| = 2hm - h|A \cap (- A)| - h + 1,\]
	and so
	\[|h (A \cup (- A))| = h |(A \cup (- A)| - h + 1.\]
	Hence it follows from Theorem \ref{hfold-direct-thm} that $A \cup (- A)$ is an arithmetic progression. If $0 \in A$, then  since $A \cup (- A)$ is an arithmetic progression, it follows that $A_{abs} = d \ast [0, m - 1]$ for some positive integer $d$. If  $0 \not\in A$, then $|A \cap (- A)|$ is an even positive integer. It is easy to see that
	\[A \cup (- A) = A_{abs} \cup (- A_{abs}).\]
	Since $A \cup (- A)$ is an arithmetic progression with common difference $d$ (say), it follows that $A_{abs}$ is also an arithmetic progression with common difference $d$. Let 
	\begin{equation}\label{ss-sumset-int-app:2-eq2}
		A_{abs} = \{a, a + d, \ldots, a + (m - 1) d\}
	\end{equation}
	for some positive integer $a$. Then
	\[A \cup (- A) = A_{abs} \cup (- A_{abs}) = \{(- a - (m - 1) d), \ldots, - a - d, - a, a,  a + d, \ldots, a + (m - 1) d.\]
	Since $A \cup (- A)$ is an arithmetic progression with common difference $d$, it follows that
	\[d = a - (-a) = 2a.\]
	Therefore, it follows from \eqref{ss-sumset-int-app:2-eq2} that
	\[A_{abs} = a \ast \{1, 3, \ldots, (2m - 1)\}.\]
	This completes the proof.
\end{proof}

\begin{corollary}\label{ss-sumset-int-app:3}
	Let $A$ be a subset of $\mathbb{Z}$ with $m \geq 3$ elements. If 
	\begin{equation} \label{ss-sumset-int-app:3-eq1}
		|2_{\pm} A| \geq
		\begin{cases}
			4m - 2, & ~\text{if}~ A \cap (- A) = \varnothing;\\
			4m - 2|A \cap (- A)|- 1, & ~\text{if}~ A \cap (- A) \neq \varnothing;
		\end{cases}
	\end{equation}
	Furthermore, the following conclusions hold:
	\begin{enumerate}
		\item If $|A \cap (- A)|$ is even and if the equality holds in \eqref{ss-sumset-int-app:3-eq1}, then $A_{abs} = d \ast \{1, 3, \ldots, \ldots, 2m - 1\}$ for some positive integer $d$.
		\item If $|A \cap (- A)|$ is odd and if the equality holds in \eqref{ss-sumset-int-app:3-eq1}, then $A_{abs} = d \ast [0, m - 1]$ for some positive integer $d$.
	\end{enumerate}
\end{corollary}

\begin{proof}
	First assume that $|A \cap (- A)| = 0$. In this case, since $2_{\pm}A = 2_{\pm}A_{abs}$, the result follows from Theorem \ref{bhanja-pandey-pos-inv1}. Now assume that $|A \cap (- A)| \geq 1$. In this case, the proof is similar to that of the Corollary \ref{ss-sumset-int-app:2}.
\end{proof}

The following lemma follows immediately from Lemma \ref{ss-nasym-group-lem:1}.
\begin{lemma}\label{ss-nasym-group-lem:1x}
	Let $H$ be a nonempty finite set of nonnegative integers, and let $G$ be an additive abelian group. Let $A$ be a subset of $G$ such that $A \cap (- A) \neq \varnothing$. Then 
	\[H_{\pm}A = H(A \cup (-A)).\] 
\end{lemma}

The following lemma expresses the signed sumset $[0, h]_{\pm}A$ as a $h$-fold signed sumset.
\begin{lemma}[{\cite[Proposition $3.3$]{bajnok2018}}]\label{ss-nasym-group-lem:1a}
	Let $h$ be positive integer. Let $A$ be nonempty subset of $G$ with $|A| \geq 2$. Then
	\[[0, h]_{\pm}A = h_{\pm} (A \cup \{0\}).\]
\end{lemma}

The following lemma follows from \cite[Proposition $3.3$ and Proposition $3.4$]{bajnok2018}. But here we give a proof based on Lemma \ref{ss-nasym-group-lem:1a} and Lemma \ref{ss-nasym-group-lem:1}.
\begin{lemma}\label{ss-nasym-group-ss-nasym-group-lem:3}
	Let $h$ be positive integer. Let $A$ be nonempty subset of $G$ with $|A| \geq 2$. Then
	\[[0, h]_{\pm}A = h(A \cup (- A) \cup \{0\}).\]
\end{lemma}

\begin{proof}
	It follows from Lemma \ref{ss-nasym-group-lem:1a} that 
	\[[0, h]_{\pm}A = h_{\pm} (A \cup \{0\}),\]
	and so it follows from Lemma \ref{ss-nasym-group-lem:1} that
	\[[0, h]_{\pm}A = h_{\pm} (A \cup \{0\}) = h(A \cup (- A) \cup \{0\}).\]
\end{proof}

The proof of the following lemma is contained in the proof of Theorem $3$ in \cite{bajnok-matzke2015}. Here, we give an alternate proof of this lemma.

\begin{lemma}\label{ss-nasym-group-lem:3}
	Let $A$ be nonempty finite subset of $G$ such that $|A| \geq 2$ and $A \cap (- A) \neq \varnothing$. If $\mathrm{sdeg}(A) \leq |A| - 2$, then there exits a set $B \subseteq G$ satisfying the following properties:
	\begin{enumerate}
		\item $|B| = |A|$,
		\item $\mathrm{sdeg}(B) \in \{|A| - 1, |A|\}$,
		\item $h_{\pm}B \subseteq h_{\pm}A$.
	\end{enumerate}	
\end{lemma}

\begin{proof}
	Since $A \cap (- A) \neq \varnothing$, it follows from Lemma \ref{ss-nasym-group-lem:1} that 
	\[h_{\pm}A = h(A \cup (- A)).\]
	Since $\mathrm{sdeg}(A) \leq |A| - 2$, there exit elements $a$, $b$ in $A$ such that $- a, - b \not\in A$. Let 
	\[A_1 = (A \cup \{- b\}) \setminus \{a\}.\]
	Then
	\[|A| = |A_1|,\]
	\[A_1 \cup (- A_1) \subseteq A \cup (- A),\]
	and
	\[\mathrm{sdeg}(A_1) = \mathrm{sdeg}(A) + 2.\]
	If $\mathrm{sdeg}(A_1) \in \{|A| - 1, |A|\}$, then we are done. Otherwise, we continue this process to get a sequence $A = A_0, A_1, \ldots, A_t, \ldots$ of subsets of $G$ with following properties:
	\begin{enumerate}
		\item $|A_i| = |A_0|$ for $i \geq 0$,
		\item $\mathrm{sdeg}(A_i) = \mathrm{sdeg}(A_{i - 1}) + 2$ for $i \geq 1$,
		\item $h_{\pm} A_{i} \subseteq h_{\pm} A_{i - 1}$ for each $i \geq 1$.
	\end{enumerate}
	Since $|A_0| \geq \mathrm{sdeg}(A_i) > \mathrm{sdeg}(A_{i - 1})$ and $|A_0|$ is finite, it follows that there exists a positive integer $t$ such that
	\[\mathrm{sdeg}(A_t) \in \{|A| - 1, |A|\}.\]
	This completes the proof.
\end{proof}

Using Theorem \ref{devos-extension} and the fact that  $|hA| \geq |A|$, we can prove the following generalization of Theorem \ref{plagne-group-h-fold-thm:1} in an additive group (not necessarily finite). 

\begin{lemma}\label{group-h-fold-lem:1}
	Let $G$ be an additive group. Let $m$ and $h$ be a positive integers. Then 
	\begin{equation}\label{group-h-fold-lem:1-eq1}
		\rho (G, m, h) \geq \min (p(G),hm - h + 1),
	\end{equation}
	with equality if and only if $m \leq p(G)$.	
\end{lemma}

\begin{proof}
	Let $A$ be a subset of $G$ with $|A| = m$. Then it follows from Theorem \ref{devos-extension} using induction that
	\[|hA| \geq \min (p(G), hm - h +1),\]
	and so
	\[\rho (G, m, h) \geq \min (p(G),hm - h + 1).\]
	This proves \eqref{group-h-fold-lem:1-eq1}. Now we prove second part. First assume that $\rho (G, m, h) = \min (p(G), hm - h + 1)$. It is easy to see that if $m \geq p(G)$, them $\min (p(G), hm - h + 1) = p(G)$. Since $|hA| \geq |A|$ for a nonempty finite set $A \subseteq G$ and $\rho (G, m, h) = \min (p(G), hm - h + 1)$, it follows that $m \leq p(G)$. 
	
	Now assume that $m \leq p(G)$.  Let $g$ be a nonzero element of $G$ such that the order of $g$ is $p(G)$. Let $H = \langle g \rangle$. Let
	\[A = [0, m - 1] \ast g.\]
	Then it is easy to see that
	\[hA = \{rg : r \in [0, hm - h + 1]\}.\]
	Hence
	\[|hA| = \min (p(G), hm - h + 1).\]
	It follows from \eqref{group-h-fold-lem:1-eq1} that
	\[\rho (G, m, h) = \min (p(G),hm - h + 1).\]
	This completes the proof.
\end{proof}

\section{Signed sumsets in groups and fields}\label{sec-signed-sum-group} 
In this section, we prove various results for the functions $\rho_{\pm}^{(s)} (G, m, h)$ and $\rho_{\pm}^{(s)} (G, m, [0, h])$ in an arbitrary abelian group $G$. First we prove some auxiliary results required for the proof of main results. 

\subsection{Signed sumset $H_{\pm} A$ in groups}\label{sec-ss-h-group-thm}

The following theorem generalizes Theorem \ref{bajnok-matzke-ss-min-group-thm} for $\rho_{\pm}(G, m, H)$.


\begin{proof}[Proof of Theorem \ref{ss-group-ext-bajnok-thm:1}]
	Let $A_0 \subseteq G$ such that $|A_0| = m$ and
	\[\rho_{\pm}(G, m, H) = |H_{\pm}A_0|.\]
	We show that there exist a set $B \in \mathcal{A}(G, m)$ such that 
	\[|H_{\pm}A_0| = |H_{\pm}B|.\]
	If $\mathrm{sdeg}(A_0) \in \{0, |A_0| - 1, |A_0|\}$, then we are done. Now assume that
	\[1 \leq \mathrm{sdeg}(A_0) \leq |A_0| - 2.\]
	Since $\mathrm{sdeg}(A_0) \leq |A_0| - 2$, it follows from Lemma \ref{ss-nasym-group-lem:3} that there exits $A_1\subseteq G$ such that
	\begin{enumerate}
		\item $|A_0| = |A_1|$.
		\item $\mathrm{sdeg}(A_1) \in \{|A_0| - 1, |A_0|\}$, and so $A_1 \in \text{Sym}(G, m) \cup \text{Nsym}(G, m)$.
		\item $h_{\pm}A_1 \subseteq h_{\pm}A_0$ for each $h \in H$. 
	\end{enumerate}
	Hence
	\[\bigcup_{h \in H}h_{\pm}A_1 \subseteq \bigcup_{h \in H}h_{\pm}A_0,\]
	and so
	\[H_{\pm}A_1 \subseteq H_{\pm}A_0.\]
	Therefore,
	\[\rho_{\pm}(G, m, H) = |H_{\pm}A_1|,\]
	and
	\[A_1 \in \text{Sym}(G, m) \cup \text{Nsym}(G, m).\]
	This completes the proof.
\end{proof}

Lemma \ref{ss-nasym-group-lem:1x} express $H$-fold signed sumset $H_{\pm}A$ as $H$-fold sumset if $A \cap (-A) \neq \varnothing$. The following theorem deals with the general case whether $A \cap (-A)$ is nonempty or not.
\begin{theorem}\label{ss-asym-group-thm:13}
	Let $H$ be a nonempty finite set of nonnegative integres, and let $A$ be a nonempty subset of an additive abelian group $G$. Then the following conclusions hold:
	\begin{enumerate}
		\item If there exists $g \in A$ such that $< g > \cap < A - g > = \{0\}$, then  
		\[|H_{\pm}A| \geq |H_{\pm}(A - g)|.\]
		\item If there exists $g \in A$ such that $< g > \cap < A \setminus \{g\}> = \{0\}$, then 
		\[|H_{\pm}A| \geq |H_{\pm}(A \cup \{0\} \setminus \{g\})|.\]
	\end{enumerate}
\end{theorem}

\begin{proof}
	First, assume that $< g > \cap < A - g > = \{0\}$ for some $g \in A$. Let 
	\[A - g = \{a_1, a_2, \ldots, a_k\}.\] 
	Then 
	\[ H_{\pm}(A - g) = \biggl \{\sum_{i = 1}^k s_ia_i : s_i \in \mathbb{Z} ~\text{for each}~ i \in [1, k], \sum_{i = 1}^k |s_i| \in H \biggl\}.\] 
	Suppose we take the set $H_{\pm}(A - g)$, so that the representation of each element has unique. In other words, if $\sum_{i = 1}^k s_ia_i$ and $\sum_{i = 1}^k u_ia_i$ are elements of $H_{\pm}(A - g)$ such that
	\[\sum_{i = 1}^k s_ia_i = \sum_{i = 1}^k u_ia_i,\]
	then
	\[s_i = u_i\] 
	for $i = 1, 2, \ldots, k$. Now we consider the set 
	\[X = \bigg \{\sum_{i = 1}^k s_i(a_i + g) : \sum_{i = 1}^k s_ia_i \in H_{\pm}(A - g) \bigg\}.\] 
	It is easy to verify that
	\[X \subseteq H_{\pm}A.\] 
	Now we want to show that $|X| = |h_{\pm}(A - g)|$. Suppose that  $\sum_{j = 1}^k s_ia_i$ and $\sum_{i = 1}^k u_ia_i$ are two distinct elements of $H_{\pm}(A - g)$ such that 
	\[\sum_{i = 1}^k s_ia_i \neq \sum_{i = 1}^k u_ia_i\]
	and 
	\[\sum_{i = 1}^k s_i(a_i + g) = \sum_{i = 1}^k u_i(a_i + g).\] 
	In this case, we get
	\[\sum_{i = 1}^k s_ia_i - \sum_{i = 1}^k u_ia_i = \bigg (\sum_{i = 1}^k (u_i - s_i) \bigg)g.\]
	Hence
	\[<g> \cap <A - g> \neq \{0\}\]
	which is a contradiction since $< g > \cap <A - g > = \{0\}$. Therefore,
	\[|H_{\pm}A| \geq |X| = |H_{\pm}(A - g)|.\] 
	
	Now, suppose that $< g > \cap < A \setminus \{g\}> = \{0\}$ for some $g \in A$.	Let 
	\[B = A \cup \{0\} \setminus \{g\} = \{a_1, a_2, \ldots, a_k\},\]
	where $a_1 = 0$.Then 
	\[H_{\pm}B = \biggl \{\sum_{i = 1}^k s_ia_i : s_i \in \mathbb{Z} ~\text{for each}~ i \in [1, k], \sum_{i = 1}^k |s_i| \in H \biggl\}.\] 
	Suppose we take the set $h_{\pm} B$, such that  each element has a unique representation. That is, if $\sum_{i = 1}^k s_ia_i$ and $\sum_{i = 1}^k u_ia_i$ are two elements of $h_{\pm}B$such that
	\[\sum_{i = 1}^k s_ia_i = \sum_{i = 1}^k u_ia_i,\]
	then 
	\[s_i = u_i\] 
	for $i = 1,2, \ldots, k$. Now we consider the set 
	\[X = \bigg \{s_1 g + \sum_{i = 2}^k s_i a_i : \sum_{i = 1}^k s_ia_i \in H_{\pm}B \bigg\}.\] 
	It is easy to verify that 
	\[X \subseteq H_{\pm}A.\] 
	Now we want to show that $|X| = |H_{\pm}B|$. Suppose that  $\sum_{j = 1}^k s_ia_i$ and $\sum_{i = 1}^k u_ia_i$ are in $h_{\pm}(A - g)$ such that $\sum_{i = 1}^k s_ia_i \neq \sum_{i = 1}^k u_ia_i$ but 
	\[s_1 g+ \sum_{i = 2}^k s_i a_i = u_1 g +\sum_{i = 2}^k u_i a_i.\] 
	If $s_1 = u_1$, then
	\[\sum_{i = 2}^k s_i a_i = \sum_{i = 2}^k u_i a_i,\]
	and so
	\[\sum_{i = 1}^k s_ia_i \neq \sum_{i = 1}^k u_ia_i\]
	which is a contradiction. Therefore,
	\[s_1 \neq u_1,\]
	and so
	\[0 \neq \sum_{i = 2}^k (s_i - u_1) a_i = (u_1 - s_1)g.\]
	Hence
	\[<g> \cap <A \setminus \{g\}> \neq \{0\}\]
	which is again a contradiction that  $< g > \cap <A - g > = \{0\}$. Therefore,
	\[|H_{\pm}A| \geq |X| = |H_{\pm}B|.\] 
	This completes the proof.
\end{proof}

Next we prove some results for the functions $\rho_{\pm}^{(s)} (G, m, h)$ and $\rho_{\pm}^{(s)} (G, m, [0, h])$. The following theorem establishes the lower bound for $\rho_{\pm}^{(s)} (G, m, h)$.


\begin{proof}[Proof of Theorem \ref{ss-nasym-group-thm:5}]
	Since $s$ is a positive integer, it follows that $A\cap (- A)\neq \varnothing$. Hence it follows from Lemma \ref{ss-nasym-group-lem:1} that 
	\[h_{\pm}A = h(A \cup (- A)).\]
	Now an application of Lemma \ref{group-h-fold-lem:1} yields
	\begin{align*}
		|h_{\pm}A| = |h(A \cup (- A))| \geq \min (p(G), 2hm - h|A \cap(- A)| - h + 1).
	\end{align*} 
	This implies \eqref{ss-nasym-group-thm:5-eq:1}. 
	
	Now we prove the second part of the theorem. First assume that
	\begin{equation}\label{ss-nasym-group-thm:5-eq:6}
		\rho_{\pm}^{(s)} (G, m, h) = \min (p(G), 2hm - hs - h + 1).
	\end{equation}
	Then
	\[\rho_{\pm}^{(s)} (G, m, h) = |h_{\pm} A_0|\]
	for some $A_0 \subseteq G$, where $|A_0| = m$ and $|A_0 \cap (- A_0)| = s$. Since $s \geq 1$, it follows from Lemma \ref{ss-nasym-group-lem:3} that
	\[\rho_{\pm}^{(s)} (G, m, h) = |h (A_0 \cup (- A_0))|.\]
	Since $|A_0 \cup (-A_0)| = 2m -s$, it follows from  \eqref{ss-nasym-group-thm:5-eq:6} and Lemma \ref{group-h-fold-lem:1} that
	\begin{multline*}
		\min (p(G), 2hm - hs - h + 1) =  \rho_{\pm}^{(s)} (G, m, h) = |h (A_0 \cup (- A))|\\
		\geq \rho (G, 2m - s, h) \geq \min (p(G), h(2m - s) - h + 1)\\ = \min (p(G), 2hm - hs - h + 1).
	\end{multline*}
	Hence 
	\[\rho (G, 2m - s, h) = \min (p(G), 2hm - hs - h + 1),\]
	and so, it follows from Lemma \ref{group-h-fold-lem:1} that $2m - s \leq p(G)$, and so $m \leq \dfrac{1}{2}(s + p(G))$.
	
	Conversely, if $m \leq \dfrac{s + p(G)}{2}$, then $2m - s \leq p(G)$. Let $g \in G$ be a nonzero element of order $p(G)$. Let $H = \langle g \rangle$, the subgroup generated by $g$. 
	Let $\phi : \mathbb{Z} \rightarrow H$ be the map defined by
	\[\phi(x) = xg.\]
	For a nonempty finite set $A \subseteq \mathbb{Z}$, let $\phi(A) = \{\phi(a) : a \in A\}$. Then $|\phi(A)| \leq |A|$ and
	\[h_{\pm}(\phi(A)) \subseteq \phi(h_{\pm}A).\]
	Consider the following cases:
	
	\noindent {\textbf{Case 1}} ($s = 2t + 1$, where $t \geq 0$). Let $A = [- t, m - t - 1] \subseteq \mathbb{Z}$, and let $B = \phi(A) \subseteq H \subseteq G$. Since $2m -s \leq p(G)$, it follows that $2(m - t - 1) < p(G)$, and so
	\[|B| = |A| = m,\]
	\[|B \cap (- B)| = |A \cap (-A)| = 2t + 1 = s,\]
	and
	\[h_{\pm}A \subseteq [- h (m - t - 1), h (m - t - 1)].\]
	Hence
	\[|h_{\pm}A| \leq 2hm - hs - h + 1.\]
	But it follows from Theorem \ref{ss-nasym-group-thm:5-eq:1} that
	\[|h_{\pm}A| \geq 2hm - hs - h + 1.\]
	Therefore,
	\[|h_{\pm}A| = 2hm - hs - h + 1.\]
	Hence
	\[~\text{either}~ |\phi(h_{\pm}A)| = 2hm - hs - h + 1 ~\text{or}~ |\phi(h_{\pm}A)| = p(G),\]
	and so
	\[|\phi(h_{\pm}A)| = \min (p(G), 2hm - hs - h + 1).\]
	Since $h_{\pm}(\phi(A)) \subseteq \phi(h_{\pm}A)$, it follows that
	\begin{equation}\label{ss-nasym-group-thm:5-eq:2}
		|h_{\pm}B| = |h_{\pm}(\phi(A))| \leq |\phi(h_{\pm}A)| = \min (p(G), 2hm - hs - h + 1).
	\end{equation}
	Since $|B| = m$, it follows from \eqref{ss-nasym-group-thm:5-eq:1} and \eqref{ss-nasym-group-thm:5-eq:2} that 
	\[|h_{\pm}B| = |\phi(h_{\pm}A)| = \min (p(G), 2hm - hs - h + 1),\]
	and so
	\begin{equation}\label{ss-nasym-group-thm:5-eq:4}
		\rho_{\pm}^{(s)} (G, m, h) \leq \min (p(G), 2hm - hs - h + 1).
	\end{equation}
	It follows from \eqref{ss-nasym-group-thm:5-eq:1} and \eqref{ss-nasym-group-thm:5-eq:4} that
	\[\rho_{\pm}^{(s)} (G, m, h) = \min (p(G), 2hm - hs - h + 1).\]
	
	\noindent {\textbf{Case 2}} ($s = 2t$, where $t \geq 1$). Let $A' = \{-(2t - 1), \ldots, - 1, 1, 3, \ldots, 2(m - t) - 1\} \subseteq \mathbb{Z}$, and let $B' = \phi (A') \subseteq H \subseteq G$. Since $2m - s \leq p(G)$, it follows that $2(m - t) - 1 < p(G)$. Now it is easy to verify that
	\[|B'| = |A'| = m,\]
	\[|B' \cap (- B')| = |A' \cap (- A')| = 2t = s,\]
	and
	\[h_{\pm}A' \subseteq [- 2h(m - t) + h, 2h(m - t) - h],\]
	Since $x \equiv y \pmod{2}$ for all $x, y \in h_{\pm} A'$, it follows that 
	\[|h_{\pm}A'| \leq 2h(m - t) - h + 1.\]
	But it follows from \eqref{ss-nasym-group-thm:5-eq:1} that
	\[|h_{\pm}A'| \geq 2h(m - t) - h + 1,\]
	and so
	\[|h_{\pm}A'| = 2h(m - t) - h + 1  = 2hm - hs - h + 1.\]
	Therefore,
	\begin{equation*}
		h_{\pm} A' =
		\begin{cases}
			2 \ast [- h(m - t) + h/2, h(m - t) - h/2], & ~\text{if}~ h \equiv 0 \pmod{2};\\
			2 \ast [- h(m - t) + \frac{h - 1}{2}, h(m - t) - \frac{h + 1}{2}] + 1, & ~\text{if}~ h \equiv 1 \pmod{2}.  
		\end{cases}
	\end{equation*}
	Hence
	\[~\text{either}~ |\phi(h_{\pm}A')| = 2hm - hs - h + 1 ~\text{or}~ |\phi(h_{\pm}A')| = p(G),\]
	and so
	\begin{equation}\label{ss-nasym-group-thm:5-eq:3}
		|h_{\pm}(\phi(A'))| \leq |\phi(h_{\pm}A')| = \min (p(G), 2hm - hs - h + 1).
	\end{equation}
	Since $|B'| = m$, it follows from \eqref{ss-nasym-group-thm:5-eq:1} and \eqref{ss-nasym-group-thm:5-eq:3} that 
	\[|h_{\pm}B'| = |h_{\pm}(\phi(A'))| = |\phi(h_{\pm}A')| = \min (p(G), 2hm - hs - h + 1),\]
	and so
	\begin{equation}\label{ss-nasym-group-thm:5-eq:5}
		\rho_{\pm}^{(s)} (G, m, h) \leq \min (p(G), 2hm - hs - h + 1).
	\end{equation}
	Therefore, it follows from \eqref{ss-nasym-group-thm:5-eq:1} and \eqref{ss-nasym-group-thm:5-eq:5} that
	\[\rho_{\pm}^{(s)} (G, m, h) = \min (p(G), 2hm - hs - h + 1).\]
	This completes the proof.
\end{proof}

The next theorem characterizes the sets $A$ for which the equality holds in \eqref{ss-nasym-group-thm:5-eq:1}. 
\begin{theorem}\label{ss-nasym-group-thm-inv:6}
	Let $h$ and $k$ be positive integers with $h \geq 2$. Let $A$ be subset of an additive abelian group $G$ with $m$ elements, and let $A \cap (- A) \neq \varnothing$. Then the following conclusions hold:
	\begin{enumerate}
			\item If $|2_{\pm}A| = 4m - 2|A \cap (- A)| - 1 < p(G) - 1$, then $A \cup (- A)$ is an arithmetic progression.
			\item If $h \geq 3$ and $|h_{\pm}A| = 2hm - h|A \cap (- A)| - h + 1 < p(G)$, then $A \cup (- A)$ is an arithmetic progression.
		\end{enumerate} 
\end{theorem}

\begin{proof}
	The proof follows from Lemma \ref{ss-nasym-group-lem:1} and Theorem \ref{kemeperman-inv-abelian-ext-thm}.  
\end{proof}

The following theorem enables us to express the function $\rho_{\pm}^{(s)} (G, m, [0, h])$ as a function $\rho_{\pm}^{(s)} (G, n, h)$ for some appropriate positive integer $n$.

\begin{theorem}\label{ss-group-thm:20}
	Let $h$, $m$ and $s$ be positive integers with $s \leq m$. Let $G$ be an additive abelian group with $p(G) \geq 3$. Then
	\begin{equation*}
		\rho_{\pm}^{(s)} (G, m, [0, h]) =
		\begin{cases}
			\rho_{\pm}^{(s)} (G, m, h), & ~\text{if}~ s \equiv 1 \pmod{2};\\
			\rho_{\pm}^{(s + 1)} (G, m + 1, h), & ~\text{if}~ s \equiv 0 \pmod{2}.  
		\end{cases}
	\end{equation*} 
\end{theorem}

\begin{proof}
	First assume that $s \equiv 1 \pmod{2}$. Let
	\[\rho_{\pm}^{(s)} (G, m, [0, h]) = |[0, h]_{\pm} A_0|\]
	for some $A_0 \subseteq G$, where $|A_0| = m$ and $|A_0 \cap (-A_0)| = s$. Then it follows from Lemma \ref{ss-nasym-group-lem:1a} that
	\[\rho_{\pm}^{(s)} (G, m, [0, h]) = |[0, h]_{\pm} A_0| = |h_{\pm} (A_0 \cup \{0\})|.\]
	Since $p(G) \geq 3$ and $s$ is odd, it follows that $0 \in A_0$, and so
	\[\rho_{\pm}^{(s)} (G, m, [0, h]) = |h_{\pm} (A_0 \cup \{0\})| = |h_{\pm}A_0| \geq \rho_{\pm}^{(s)} (G, m, h).\]
	Now we show that $\rho_{\pm}^{(s)} (G, m, [0, h]) = \rho_{\pm}^{(s)} (G, m, h)$. Suppose that $\rho_{\pm}^{(s)} (G, m, [0, h]) > \rho_{\pm}^{(s)} (G, m, h)$. Let
	\[\rho_{\pm}^{(s)} (G, m, h) = |h_{\pm} A_1|\]
	for some $A_1 \subseteq G$ with $|A_1| = m$ and $|A_1 \cap (-A_1)| = s$. Since $p(G) \geq 3$ and $s$ is odd, it follows that $0 \in A_1$. Hence $h_{\pm} A_1 = [0, h]_{\pm} A_1$, and so
	\[\rho_{\pm}^{(s)} (G, m, h) = |h_{\pm} A_1| = |[0, h]_{\pm} A_1| \geq \rho_{\pm}^{(s)} (G, m, [0, h]),\]
	which is a contradiction. Therefore, 
	\[\rho_{\pm}^{(s)} (G, m, [0, h]) = \rho_{\pm}^{(s)} (G, m, h).\]
	
	Now assume that  $s \equiv 0 \pmod{2}$. Let
	\[\rho_{\pm}^{(s)} (G, m, [0, h]) = |[0, h]_{\pm} A_2|,\]
	for some $A_2 \subseteq G$,	where $|A_2| = m$ and $|A_2 \cap (-A_2)| = s$. It follows from Lemma \ref{ss-nasym-group-lem:1a} that
	\[\rho_{\pm}^{(s)} (G, m, [0, h]) = |[0, h]_{\pm} A_2| = |h_{\pm} (A_2 \cup \{0\})|.\]
	Since $p(G) \geq 3$, it follows that $0 \not\in A_2$, and so
	\[\rho_{\pm}^{(s)} (G, m, [0, h]) = |h_{\pm} (A_2 \cup \{0\})|  \geq \rho_{\pm}^{(s + 1)} (G, m + 1, h).\]
	Suppose that $\rho_{\pm}^{(s)} (G, m, [0, h]) > \rho_{\pm}^{(s + 1)} (G, m + 1, h)$. Let 
	\[\rho_{\pm}^{(s + 1)} (G, m + 1, h) = |h_{\pm} A_3|\]
	for some $A_3 \subseteq G$ with $|A_3| = m + 1$ and $|A_3 \cap (-A_3)| = s + 1$. Since $p(G) \geq 3$ and $s$ is an even, it follows that $0 \in A_3$. Hence
	\[\rho_{\pm}^{(s + 1)} (G, m + 1, h) = |h_{\pm} A_3| = |[0, h]_{\pm} (A_3 \setminus \{0\})| \geq \rho_{\pm}^{(s)} (G, m, [0, h]),\]
	which is a contradiction. Therefore, 
	\[\rho_{\pm}^{(s)} (G, m, [0, h]) = \rho_{\pm}^{(s + 1)} (G, m + 1, h).\]
	This completes the proof.
\end{proof}

\begin{theorem}\label{ss-group-thm:2}
	Let $G$ be an additive abelian group. Let $h$ and $m \geq 2$ be positive integers. Then
	\begin{enumerate}
		\item $\rho_{\pm}(G, m, [0, h]) = \min \{|[0, h]_{\pm}A| : A \in \text{Sym}(G, m)\}$.
		\item $\rho_{\pm}(G, m, [0, h]) = \min \{|h(A \cup \{0\})| : A \in \text{Sym}(G, m)\}$.    
		\item If $m$ is an odd integer, then 
		\[\rho_{\pm}(G, m, [0, h]) = \min \{|hA|: A \in \text{Sym}(G, m), 0 \in A\}.\]
		\item If $n$ is an odd integer, then
		\[\rho_{\pm}(G, m, [0, h]) = \min \{|hA|: A \in \text{Sym}(G, 2 \lfloor m / 2 \rfloor + 1), 0 \in A\}.\]
	\end{enumerate}
\end{theorem}

\begin{proof}
	Let $A_0 \subseteq G$ such that $|A_0| = m$ and
	\[\rho_{\pm}(G, m, [0, h]) = |[0, h]_{\pm}A_0|.\]
	Then it follows from Lemma \ref{ss-nasym-group-ss-nasym-group-lem:3} that
	\[\rho_{\pm}(G, m, [0, h]) = |[0, h]_{\pm}A_0|= |h(A_0 \cup (- A_0) \cup \{0\})|.\]
	If $A_0 \cap (- A_0) = \varnothing$, then for every $ a \in A_0$ such that $- a \not\in A_0$. Let 
	\[A_1 = (A_0 \cup \{- a\}) \setminus\{b\},\]
	where $a, b \in A_0$ with $a \neq b$. It is easy to verify that
	\[|A_1| = |A_0|,\]
	\[A_1 \cap (- A_1) \neq \varnothing,\]
	\[\mathrm{sdeg}(A_1) = \mathrm{sdeg}(A_0) + 2,\]
	and
	\[A_1 \cup (- A_1) \cup \{0\} \subseteq A_0 \cup (- A_0) \cup \{0\}.\]
	Since $\rho_{\pm}(G, m, [0, h]) = |[0, h]_{\pm}A_0|$, $|A_1| = m$, and $A_1 \cup (- A_1) \cup \{0\} \subseteq A_0 \cup (- A_0) \cup \{0\}$, it follows from Lemma \ref{ss-nasym-group-ss-nasym-group-lem:3} that
	\[\rho_{\pm}(G, m, [0, h]) = |h(A_1 \cup (- A_1) \cup \{0\})| = |[0, h]_{\pm}A_1|.\]
	If $|A_1| = 2$, then $A_1$ is a symmetric set, so we are done. Assume that $|A_1| \geq 3$. It follows from Lemma \ref{ss-nasym-group-lem:3} that there exits a positive $t$ such that
	\[|A_t| = |A_0|,\]
	\[\mathrm{sdeg}(A_t) \in \{|A_0| - 1, |A_0|\},\]
	and
	\[\rho_{\pm}(G, m, [0, h]) = |h(A_t \cup (- A_t) \cup \{0\})| = |[0, h]_{\pm}A_t|.\]
	Since $\mathrm{sdeg}(A_t) \in \{|A_0| - 1, |A_0|\}$, it follows that
	\[A_t \in \text{Sym}(G, m) \cup \text{Nsym}(G, m).\]
	If $A_t \in \text{Sym}(G, m)$, then we are done. Now assume that $A_t \in \text{Nsym}(G, m)$, then there $a \in A_t$ such that $A \setminus \{a\}$ is a symmetric set. Note that $a \neq 0$.
	
	First, we assume that $0 \in A_t$. Let
	\[B = (A_t \cup \{- a\}) \setminus \{0\}.\]
	Then
	\[B \cup (- B) \cup \{0\} \subseteq A_t \cup (- A_t) \cup \{0\},\]
	and
	\[B \in \text{Sym}(G, m).\]
	Hence
	\[h(B \cup (- B) \cup \{0\}) \subseteq h(A_t \cup (- A_t) \cup \{0\}).\]
	By applying Lemma \ref{ss-nasym-group-ss-nasym-group-lem:3}, we get
	\[\rho_{\pm}(G, m, [0, h]) = |h(B \cup (- B) \cup \{0\})| = |[0, h]_{\pm}B|,\]
	where $B \in \text{Sym}(G, m)$.
	
	Now assume that $0 \not\in A_t$. Let
	\[C = (A_t \cup \{0\}) \setminus \{a\}.\]
	Then
	\[C \cup (- C) \cup \{0\} \subseteq A_t \cup (- A_t) \cup \{0\},\]
	and
	\[C \in \text{Sym}(G, m).\]
	Hence
	\[h(C \cup (- C) \cup \{0\}) \subseteq h(A_t \cup (- A_t) \cup \{0\}).\]
	By applying Lemma \ref{ss-nasym-group-ss-nasym-group-lem:3}, we get
	\[\rho_{\pm}(G, m, [0, h]) = |h(C \cup (- C) \cup \{0\})| = |[0, h]_{\pm}C|,\]
	where $C \in \text{Sym}(G, m)$. Therefore,
	\[\rho_{\pm}(G, m, [0, h]) = \min \{|[0, h]_{\pm}A| : A \in \text{Sym}(G, m)\}.\]
	This proves the first part of the theorem.
	
	By applying Lemma \ref{ss-nasym-group-ss-nasym-group-lem:3}, we get
	\[\rho_{\pm}(G, m, [0, h]) = \min \{|h(A \cup \{0\})| : A \in \text{Sym}(G, m)\}.\]
	This concludes the proof of the second part.
	
	Now we assume that $m$ is an odd integer. Let $A_0 \in \text{Sym}(G, m)$ such that
	\[\rho_{\pm}(G, m, [0, h]) = |h(A_0 \cup \{0\})|.\]
	If $0 \in A_0$, then we are done. Now, we assume that $0 \not\in A_0$. Since $A_0 \in \text{Sym}(G, m)$ and $m$ is odd, it follows that there exists $a_0 \in A$ such that
	\[A_0 = B_0 \cup \{a_0\},\]
	and
	\[B_0 = - B_0,\]
	where $a_0 = - a_0$ with $a_0 \not\in B_0$. Let $A_1 = B_0 \cup \{0\}$. Then
	\[|A_1| = m,\]
	and
	\[h A_1 = h(B_0 \cup \{0\}) \subseteq h(A_0 \cup \{0\}).\]
	Hence
	\begin{equation}\label{eq:7}
		|h A_1| = |h(A_1 \cup \{0\})| \leq |h(A_0 \cup \{0\})|.
	\end{equation}
	Since $\rho_{\pm}(G, m, [0, h]) = |h(A_0 \cup \{0\})|$ and $|A_0| = |A_1| = m$, it follows from \eqref{eq:7} that 
	\[\rho_{\pm}(G, m, [0, h]) = |h A_1|,\]
	and
	\[0 \in A_1.\]
	Therefore,
	\[\rho_{\pm}(G, m, [0, h]) = \min \{|hA| : A \in \text{Sym}(G, m), 0 \in A\}.\]
	This concludes the proof of the second part.
	
	Now, we assume that $n$ is an odd integer. Let $A_0 \in \text{Sym}(G, m)$ such that
	\[\rho_{\pm}(G, m, [0, h]) = |h(A_0 \cup \{0\})|.\]
	If $m$ is even, then since $n$ is an odd integer and $A_0 \in \text{Sym}(G, m)$, it follows that
	\[0 \notin A_0,\]
	and so
	\[|A_0 \cup \{0\}| = m + 1.\]
	Hence
	\[\rho_{\pm}(G, m, [0, h]) = \min \{|hA|: A \in \text{Sym}(G, m + 1), 0 \in A\}.\]
	If $m$ is odd, then since $n$ is odd, and $A_0 \in \text{Sym}(G, m)$, it follows that
	\[0 \in A_0,\]
	and so
	\[|A_0 \cup \{0\}| = m.\]
	Hence
	\[\rho_{\pm}(G, m, [0, h]) = \min \{|hA|: A \in \text{Sym}(G, m), 0 \in A\}.\]
	Therefore,
	\[\rho_{\pm}(G, m, [0, h]) = \min \{|hA|: A \in \text{Sym}(G, 2 \lfloor m / 2 \rfloor + 1), 0 \in A\}.\]
	This concludes the proof of the third part.
\end{proof}

\subsection{Size of the signed sumset $h_{\pm}A$ in a field} \label{sec-hfold-signed-sum-field}
Obtaining the optimal lower bound on the size of the signed sumset $h_{\pm}A$ in a field seems very difficult. In this subsction, we use the polynomial method to obtain few more results on the size of $h$-fold signed sumsets $h_{\pm}A$ for $h \in \{2, 3, 4\}$. The following theorem due to Alon, Nathanson and Ruzsa \cite{alon2} will be useful for the proofs.

\begin{theorem}\label{thm:A_2}
	Let $h$ be a positive integer, and let $p(x_1, \ldots, x_h) \in \mathbb{F}[x_1, \ldots, x_h]$ be a polynomial defined over a field $\mathbb{F}$. Let $K$ be a nonnegative integer, and let $A_1, \ldots, A_h$ be nonempty finite subsets of $\mathbb{F}$ such that $\sum_{i = 1}^h |A_i| =  K + h + ~\mathrm{deg} (p)$, where $\mathrm{deg}(p)$ denote the degree of the polynomial $p(x_1, \ldots, x_h)$. Let
	\[S = \{a_1 + \cdots + a_h : a_i \in A_i ~\text{for each}~ i \in [1, h],\ p(a_1, \ldots, a_h)\neq 0\}.\]
	If the coefficient of $x_1^{|A_1| - 1} \ldots x_h^{|A_h| - 1}$ in the polynomial $(x_1 + \cdots + x_h)^K p(x_1, \ldots, x_h)$ is non-zero, then 
	\[|S|\geq K + 1.\]
\end{theorem}

In this section, we proved the following results
\begin{theorem}\label{ss-asym-prime-thm:9} 
	Let $A$ be a nonempty finite subset of  a field $\mathbb{F}$ with $|A| = k$ and $A \cap (- A) = \varnothing$. Then
	\begin{equation*}
		|2_{\pm}A| \geq
		\begin{cases}
			4k - 2, & \mbox{if } p(\mathbb{F}) \geq 4k - 1; \\
			p(\mathbb{F}) - 1, & \mbox{otherwise}.
		\end{cases}
	\end{equation*}
\end{theorem}

\begin{proof}
	First assume that $p(\mathbb{F}) \geq 4k - 1$. Let $A^* = A \cup (- A)$. Then $|A^*| = 2k$. In the Theorem \ref{thm:A_2}, let $h = 2$, $A_1 = A_2 = A^*$ and $p(x_1, x_2) = x_1 + x_2$. Then we get $K= 4k - 3$ and \[2_{\pm}A = \{a_1 + a_2 : a_1, a_2 \in A^* ~\text{with}~ a_1 + a_2\neq 0\}.\]
	Since the coefficient of $x_1^{2k - 1}x_2^{2k - 1}$ in the polynomial $(x_1 + x_2)^{K + 1}$ is $\binom{4k - 2}{2k - 1} = \frac{(4k - 2)!}{((2k - 1)!)^2}$, which is nonzero modulo $p(\mathbb{F})$, it follows from Theorem \ref{thm:A_2} that
	\[|2_{\pm}A| \supseteq |\{a_1 + a_2 : a_1, a_2 \in A^* ~\text{with}~ a_1 + a_2 \neq 0\}| \geq K + 1 = 4k - 2.\]
	
	Now assume that	$p(\mathbb{F}) \leq 4k - 3$. Then $p(\mathbb{F}) = 4k - 3 - 2l$ for some integer $l$. We choose a subset of $A'$ of $A^*$ such that $|A'| = 2k - l - 1$. Let $A_1 = A_2 = A'$. The coefficient of $x_1^{2k - l -2}x_2^{2k - l - 2}$ in the polynomial $(x_1 + x_2)^{K + 1}$ is 
	\[\binom{4k - 2l - 4}{2k - l - 2} = \frac{(4k - 2l - 4)!}{((2k - l - 2)!)^2}\]
	which is nonzero modulo $p(\mathbb{F})$. Therefore, by applying Theorem \ref{thm:A_2} on the sets $A_1 = A_2 = A'$, we get
	\[|2_{\pm}A| \geq |\{a_1 + a_2 : a_1, a_2 \in A' ~\text{with}~ a_1 + a_2 \neq 0\}| \geq K + 1 = 4k - 4 - 2l = p(\mathbb{F}) - 1.\] 
	This completes the proof.
\end{proof}

\begin{theorem}\label{ss-asym-prime-thm:10}
	Let $\mathbb{F}$ be a field such that $p(\mathbb{F}) > 6k - 6$ with $p(\mathbb{F}) \neq 8k - 7$. Let $A$ be a nonempty subset of $\mathbb{F}$ with $k \geq 2$ elements such that $A \cap (- A) = \varnothing$. Then
	\[|3_{\pm}A| \geq 6k - 5.\]
\end{theorem}

\begin{proof}
	Let $A^* = A\cup (- A)$. Then $|A^*| = 2k$. Let $h = 3$, $A_1 = A_2 = A_3 = A^*$ and $p(x_1, x_2, x_3) = (x_1 + x_2)(x_1 + x_3)(x_2 + x_3)$. We apply Theorem \ref{thm:A_2} for $h = 3$. In this case, $K = 6k - 6$ and 
	\[3_{\pm}A \supseteq \{a_1 + a_2 + a_3 : a_1, a_2, a_3 \in A^* ~\text{with}~ p(a_1, a_2, a_3) \neq 0\}.\]
	Now we have 	
	\begin{align*}  
		(x_1 &+ x_2 + x_3)^K p(x_1, x_2, x_3)\\
		&= (x_1 + x_2 + x_3)^{6k - 6}(x_1 + x_2)(x_1 + x_3)(x_2 + x_3)\\
		&= (x_1 + x_2 + x_3)^{6k - 6}(x_1x_3^2 + x_2x_3^2 + x_1x_2^2 + x_2^2x_3 + x_1^2x_2 + x_1^2x_3 + 2x_1x_2x_3),
	\end{align*}
	and so
	\begin{multline*}
		(x_1 + x_2 + x_3)^K p(x_1, x_2, x_3) \\
		=(x_1 + x_2 + x_3)^{6k - 6}(x_1x_3^2 + x_2x_3^2 + x_1x_2^2 + x_2^2x_3 + x_1^2x_2 + x_1^2x_3)\\
		+ 2(x_1 + x_2 + x_3)^{6k - 6}(x_1x_2x_3).
	\end{multline*}
	Hence coefficient of $x_1^{2k - 1}x_2^{2k - 1}x_3^{2k - 1}$ in the polynomial $(x_1 + x_2 + x_3)^K p(x_1, x_2, x_3)$ is
	\[\frac{2(6k - 6)!}{(2k - 2)!^3} + \frac{6(6k - 6)!}{(2k - 1)!(2k - 2)!(2k - 3)!} = \frac{2(8k - 7)(6k - 6)!}{(2k - 1)!(2k - 2)!^2},\] 
	which is nonzero modulo $p(\mathbb{F})$ when $p(\mathbb{F}) > 6k - 6$ and $p(\mathbb{F}) \neq 8k - 7$. Therefore, it follows from Theorem \ref{thm:A_2} that
	\[|3_{\pm}A| \geq K + 1 = 6k - 5.\]
	This completes the proof.
\end{proof}

\begin{theorem}\label{ss-asym-prime-thm:11}
	Let $A$ be a nonempty subset of a field $\mathbb{F}$ with $k \geq 2$ elements, and let $A \cap (- A) = \varnothing$. Suppose that one of the following conditions is satisfied:
	\begin{enumerate}
		\item $32k^2 - 64k + 31 < p(\mathbb{F})$,
		\item $8k - 10 < p(\mathbb{F}) < 32k^2 - 64k + 31$,
		\item $8k - 10 < p(\mathbb{F})$ and $p(\mathbb{F}) \equiv \pm 3 \pmod 8$.
	\end{enumerate}
	Then
	\begin{equation}\label{four-fold-signed-sum-eq1}
		|4_{\pm}A| \geq 8k -9.
	\end{equation} 
	If $l$ is the least positive integer such that $8k - 10 - 4l < p(\mathbb{F}) < 8k -10 - 4(l-1)$ and if $p(\mathbb{F}) \equiv \pm 3 \pmod 8$, then
	\begin{equation}\label{four-fold-signed-sum-eq2}
		|4_{\pm}A| \geq 8k - 9 - 4l.
	\end{equation}
\end{theorem}

\begin{proof}
	First we prove \eqref{four-fold-signed-sum-eq1}. Let $A^* = A \cup (- A)$. Then $|A^*| = 2k$. We apply Theorem \ref{thm:A_2} for $h = 4$, $A_1 = A_2 = A_3 = A_4 = A^*$ and the polynomial
	\[p(x_1, x_2, x_3, x_4) = \prod\limits_{1 \leq i < j \leq 4}(x_i + x_j).\]
	In this case, $K = 8k - 10$ and 
	\[4_{\pm}A \supseteq \{a_1 + a_2 + a_3 + a_4 : a_1, a_2, a_3 , a_4 \in A^*, p(a_1, a_2, a_3, a_4) \neq 0\}.\]
	Clearly,
	\begin{align*}
		p(x_1, x_2, x_3, x_4) & =\prod\limits_{1 \leq i < j \leq 4}(x_i + x_j)\\
		& = x_2x_3^2x_4^3 + x_1x_3^2x_4^3 + x_2^2x_3x_4^3 + 2x_1x_2x_3x_4^3 + x_1^2x_3x_4^3 + x_1x_2^2x_4^3\\
		&+ x_1^2x_2x_4^3 + x_2x_3^3x_4^2 + x_1x_3^3x_4^2 + 2x_2^2x_3^2x_4^2 + 4x_1x_2x_3^2x_4^2 + 2x_1^2x_3^2x_4^2\\
		&+ x_2^3x_3x_4^2 + 4x_1x_2^2x_3x_4^2 + 4x_1^2x_2x_3x_4^2 + x_1^3x_3x_4^2 + x_1x_2^3x_4^2 + 2x_1^2x_2^2x_4^2\\
		&+ x_1^3x_2x_4^2 + x_2^2x_3^3x_4 + 2x_1x_2x_3^3x_4 + x_1^2x_3^3x_4 + x_2^3x_3^2x_4 + 4x_1x_2^2x_3^2x_4\\
		&+ 4x_1^2x_2x_3^2x_4 + 2x_1^3x_2x_3x_4 + x_1^2x_2^3x_4 + x_1x_2^2x_3^3 + x_1^2x_2x_3^3 + x_1x_2^3x_3^2\\
		&+ 2x_1^2x_2^2x_3^2 + x_1^3x_2x_3^2 + x_1^2x_2^3x_3 + x_1^3x_2^2x_3 + x_1^3x_2^2x_4 + x_1^3x_3^2x_4\\
		&+ 2x_1x_2^3x_3x_4 + 4x_1^2x_2^2x_3x_4.
	\end{align*}
	Thus the polynomial $p(x_1, x_2, x_3, x_4) (x_1 + x_2 + x_3 + x_4)^{8k-10}$ is the sum of several terms, where each term is a product of a monomial and the polynomial $(x_1 + x_2 + x_3 + x_4)^{8k-10}$. Now to compute the coefficient of $x_1^{2k - 1}x_2^{2k - 1}x_3^{2k - 1}x_4^{2k - 1}$ in the polynomial $(x_1 + x_2 + x_3 + x_4)^K p(x_1, x_2, x_3, x_4)$, we compute the coefficient of each such term separately and collect the similar coefficients together. In this way, we get the following coefficient of $x_1^{2k - 1}x_2^{2k - 1}x_3^{2k - 1}x_4^{2k - 1}$ in $p(x_1, x_2, x_3, x_4) (x_1 + x_2 + x_3 + x_4)^{8k-10}$:
	\begin{multline*}
		\frac{8(8k - 10)!}{(2k - 1)!{(2k - 3)!}^3}  + \frac{24(8k - 10)!}{{(2k - 2)!}^2{(2k - 3)!}^2} + \frac{8(8k - 10)!}{{(2k - 2)!}^3(2k - 4)!} \\ 
		+ \frac{24(8k - 10)!}{(2k - 1)!(2k - 2)!(2k - 3)!(2k - 4)!}.
	\end{multline*}
	By simplifying the above expression, we get the required coefficient as the following expression:
	\begin{equation}\label{monomial-coef}
		\frac{8(8k - 10)!(32k^2 - 64k + 31)}{(2k - 1)!{(2k - 2)!}^2(2k - 3)!}.
	\end{equation}
	Note that $8k - 10 \leq 32k^2 - 64k + 31$ for all $k \geq 2$. Now we consider the following cases:
	
	\noindent {\textbf{Case 1}} ($32k^2 - 64k + 31 < p(\mathbb{F})$).
	In this case, the coefficient \eqref{monomial-coef} is nonzero modulo $p(\mathbb{F})$. Therefore, it follows from Theorem \ref{thm:A_2} that
	\[|4_{\pm}A| \geq K + 1 = 8k - 9.\]
	
	\noindent {\textbf{Case 2}} ($8k - 10 < p(\mathbb{F}) < 32k^2 - 64k + 31$). In this case also, the coefficient \eqref{monomial-coef} is nonzero modulo $p(\mathbb{F})$. Therefore, it follows from Theorem \ref{thm:A_2} that
	\[|4_{\pm}A| \geq K + 1 = 8k - 9.\]
	
	\noindent {\textbf{Case 3.}} ($8k - 10 < p(\mathbb{F})$ and $p(\mathbb{F}) \equiv {\pm} 3 \pmod{8}$). In this case, we show that $32k^2 - 64k + 31 \not \equiv 0 \pmod {p(\mathbb{F})}$. Suppose that $32k^2 - 64k + 31 \equiv 0 \pmod {p(\mathbb{F})}$. Since $32k^2 - 64k + 31 = 2(4k - 4)^2 - 1$, it follows that 
	\[2t^2 \equiv 1 \pmod{p(\mathbb{F})},\]
	where $t = (4k - 4)^2$. Then $t \neq 0 \pmod{p(\mathbb{F})}$, and so
	\[2 \equiv t'^2 \pmod{p(\mathbb{F})},\]
	where $tt' \equiv 1 \pmod{p(\mathbb{F})}$. Thus $2$ is a quadratic residue modulo $p(\mathbb{F})$ which implies that
	\[p(\mathbb{F}) \equiv \pm 1 \pmod{8},\]
	which is a contradiction. Thus, we have
	\[32k^2 - 64k + 31 \not \equiv 0 \pmod {p(\mathbb{F})}.\]
	Hence it follows from Theorem \ref{thm:A_2} that
	\[|4_{\pm} A| \geq K + 1 = 8k - 9.\]
	The above three cases implies the inequality \eqref{four-fold-signed-sum-eq1}.
	
	To prove \eqref{four-fold-signed-sum-eq2}, let $A'$ be subsets of $A \cup (-A)$ with $|A'| = 2k - l$. Let $A_1 = A_2 = A_3 = A_4 = A'$. We apply Theorem \ref{thm:A_2} for $h = 4$, $A_1 = A_2 = A_3 = A_4 = A^*$ and the same polynomial
	\[p(x_1, x_2, x_3, x_4) = \prod\limits_{1 \leq i < j \leq 4}(x_i + x_j).\]
	In this case, $K = 8k - 10 - 4l$. By the similar method used to compute the coefficient \eqref{monomial-coef}, we compute the coefficient of $x_1^{2k - 1 -l}x_2^{2k - 1-l}x_3^{2k - 1-l}x_4^{2k - 1-l}$ in the polynomial $(x_1 + x_2 + x_3 + x_4)^K p(x_1, x_2, x_3, x_4)$ which is equal to the following expression:
	\begin{equation}\label{monomial-coef2}
		\frac{8(8k - 10 - 4l)!(32k^2 - 64k - 32kl + 32l + 8l^2 + 31)}{(2k - 1 - l)!{(2k - 2 - l)!}^2(2k - 3 - l)!}.
	\end{equation}
	Note that
	\begin{align*}
		32k^2 - 64k - 32kl + 32l + 8l^2 + 32 = 2(4k - 2l - 4)^2.
	\end{align*}
	Hence an argument similar to that in Case $3$ implies that
	\[32k^2 - 64k - 32kl + 32l + 8l^2 + 31 \not \equiv 0 \pmod 8.\] 
	Hence it follows from Theorem \ref{thm:A_2} that
	\[|4_{\pm}A| \geq K + 1 = 8k - 4l - 9.\]
	This completes the proof.
\end{proof}

\section{Restricted signed sumsets in groups and fields}\label{sec-res-signed-sum-group}
In this section, we prove lower bounds on the size of the restricted $h$-fold signed sumset $h_{\pm}^\wedge A$ in a arbitrary field $\mathbb{F}$, and the restricted $[0, h]$-fold signed sumset $[0, h]_{\pm}^\wedge A$ in additive abelian groups.

\subsection{Restricted $h$-fold signed sumset $h_{\pm}^{\wedge} A$ in fields}
In this subsection, we prove some results for the restricted signed sumsets in a field $\mathbb{F}$. For the proof, we will need the following results. The first theorem is due to Dias da Silva and Hamidoune \cite{dias}. The theorem was later reproved by Alon, Nathanson and Ruzsa using the polynomial method \cite{alon1, alon2}.
\begin{theorem}\label{alon-nath-ruzsa}
	Let $h$ and $m$ be integers such that $2 \leq h \leq m$. Let $A$ be a subset of a field $\mathbb{F}$ with $m$ elements. Then
	\[|h ^\wedge A| \geq \min (p(\mathbb{F}), hm - h^2 + 1).\]	
\end{theorem}

The next theorem is due to Liu and Sun \cite{liu2002}.

\begin{theorem}[{\cite[Liu and Sun]{liu2002}}]\label{liu-sun-poly-thm-2002}
	Let $k$, $m$ and $h$ be positive integers such that $K = (k - 1)h - (m + 1) \frac{h(h-1)}{2}$ is nonnegative. Let $p(\mathbb{F}) \geq \max (K, m, h - 1)$. Let $A_1, \ldots, A_h$ be subsets of $\mathbb{F}$ such that $|A_i| \geq k - h + i$ for $i \in [1, h]$. Let $P_1, \ldots, P_h \in \mathbb{F}[t]$ be monic polynomials of degree $m$. Let
	\[C = \{a_1 + \cdots + a_h : a_{i} \in A_i, P_i(a_i) \neq P_j(a_j) ~\text{if}~ i\neq j \}.\]
	Then
	\[|C| \geq K + 1.\] 	
\end{theorem}


%

\begin{proof}[Proof of Theorem \ref{rss-field-thm:3}]
	If $\min(p(\mathbb{F}), \theta + 1) \leq hk - h^2 + 1$, then since $h^{\wedge}A \subseteq h_{\pm}^{\wedge} A$, the result follows trivially from Theorem \ref{alon-nath-ruzsa}. Now assume that $hk - h^2 + 1 < \min (p(\mathbb{F}), \theta + 1)$. We consider the following two cases:
	
	\noindent {\textbf{Case 1}} ($hk - h^2 + 1  < \theta + 1 \leq p(\mathbb{F})$). Let 
	\[A_i = A \cup (- A)\]
	for each $i \in [1, h]$. Let $P_1(t), \ldots, P_h(t) \subseteq \mathbb{F}[t]$ such that
	\[P_1(t) = \cdots = P_h(t) = t^2.\]
	Then the value of $K$ defined in Theorem \ref{liu-sun-poly-thm-2002} is given by
	\[K= h|A \cup (- A)| - \frac{h(3h - 1)}{2} = \theta.\]
	Let 
	\[S = \{x_1 + \cdots + x_h : x_{i} \in A \cup (- A), P_i(a_i) \neq P_j(a_j) ~\text{if}~ i \neq j \}.\]
	We show that $S \subseteq h_{\pm}^{\wedge} A$. Let $x_1 + \cdots + x_h \in S$. Then
	\[x_i \in A \cup (- A)\]
	for each $i \in [1, h]$. Hence there exist $s_1, \ldots, s_h \in \{- 1, 1\}$ and $a_1, \ldots, a_h \in A$ such that
	\[x_i = s_ia_i\]
	for each $i \in [1, h]$. Since $x_i^2 \neq x_j^2$ for each $i \neq j$, it follows that $a_i^2 \neq a_j^2$ for each $i \neq j$. Hence $a_1, \ldots, a_h$ are distinct element of $A$. Thus
	\[x_1 + \cdots + x_h = s_1a_1 + \cdots + s_ha_h \in h_{\pm}^{\wedge} A.\]
	Therefore,
	\[S \subseteq h_{\pm}^{\wedge} A,\]
	and so it follows from Theorem \ref{liu-sun-poly-thm-2002} that
	\begin{align*}
		|h_{\pm}^{\wedge} A| \geq |S| &\geq K + 1 = \theta + 1.\\	
	\end{align*}

	\noindent {\textbf{Case 2}} ($hk - h^2 + 1  < p(\mathbb{F}) < \theta + 1$). In this case, it is given that $\ell$ is the least positive integer such that 
	\[\theta - \ell h + 1 \leq p(\mathbb{F}) < \theta - (\ell - 1)h + 1.\]	
	Let $r$ be the nonnegative integer such that $p(\mathbb{F}) = \theta - \ell h + r + 1$. Then $0 \leq r < h$. Now if $\theta - \ell h + 1 \leq hk - h^2 + 1 < p(\mathbb{F})$, then it follows from Theorem \ref{alon-nath-ruzsa} that
	\[|h_{\pm}^{\wedge} A| \geq hk - h^2 + 1 = \max(hk - h^2 + 1, \theta - \ell h + 1).\]
	Now assume that $hk - h^2 + 1 < \theta - \ell h + 1 \leq  p(\mathbb{F})$. Note that $|A \cup (-A)| = 2k - |A \cap (-A)|$. Since  $hk - h^2 + 1 < p(\mathbb{F})$ and $p(\mathbb{F}) = \theta - \ell h + r + 1$, it follows that
	\begin{align}\label{rss-field-thm:3-eq1}
		h(|A \cup (-A)| - \ell) = 2hk - h|A \cap (- A)| - \ell h  &= \theta + \frac{h(3h - 1)}{2} -  \ell h \notag \\
		&\geq \frac{h(3h - 1)}{2} + hk - h^2 + 1 \geq h^2 + h.
	\end{align}
	Let 
	\[k' = |A \cup (-A)| - \ell =  2k  - |A \cap (- A)| - \ell.\]
	Then it follows from \eqref{rss-field-thm:3-eq1} that
	\[hk' =  2hk  - h|A \cap (- A)| - \ell h \geq h^2 + h,\]
	and so
	\[k' \geq h + 1.\]
	Let $A_1, \ldots, A_h$  be subsets of $A \cup (- A)$ such that 
	\[|A_i| \geq k' - h + i\]
	for $i \in [1, h]$. Then $A_1, \ldots, A_h$ are nonempty subsets $A \cup (- A)$. Let 
	\[T = \{a_1 + \cdots + a_h : a_{i} \in A_i, i = 1, \ldots, h, a_i^2 \neq a_j^2 ~\text{if}~ i \neq j \}.\]
	Then it is easy to verify that
	\[T \subseteq h_{\pm}^{\wedge} A,\]
	and so
	\[|h_{\pm}^{\wedge} A| \geq |T|.\]
	By applying Theorem \ref{liu-sun-poly-thm-2002}, we get
	\[|h_{\pm}^{\wedge} A| \geq (k' - 1) h - \frac{3h(h - 1)}{2} + 1 = \theta - \ell h + 1 = \max (hk - h^2 + 1, \theta - \ell h + 1).\]
	Thus in both the cases, we have
	\[|h_{\pm}^{\wedge} A| \geq \max (hk - h^2 + 1, \theta - \ell h + 1).\]
	This completes the proof.
\end{proof}	

The following example shows that we can not expect the lower bound $\min (p, 2hk - h^2 + 1)$ on the size of $h_{\pm}^\wedge A$ in $\mathbb{Z}_p$ if $A \cap (-A) = \varnothing$.
\begin{example}
	Let $h = 2, k = 5$, and let $A = \{1, 2, 3, 4, 5\} \subseteq \mathbb{Z}_{17}$. Then a computation shows that
	\[|2_{\pm}^\wedge A| = 16 = 2hk - \frac{h(3h - 1)}{2} + 1 = \min (p, 2hk - \frac{h(3h - 1)}{2} + 1).\]
	Here
	\[|2_{\pm}^\wedge A| = 16 < \min (p, 2hk - h^2 + 1).\]
\end{example}

The next example shows that the lower bound in Theorem \ref{rss-field-thm:3} is sharp even if $A \cap (-A) \neq \varnothing$.
\begin{example}
	Let $h = 3, k = 9$, and let $A = \{0, 1, 3, 5, 7, 9, 11, 13, 15\} \subseteq \mathbb{Z}_{41}$. Then $A \cap (- A)| = \{0\}$, and a computation shows that
	$3_{\pm}^\wedge A = \mathbb{Z}_{41} \setminus \{0\}$,
	and so	
	\[|3_{\pm}^\wedge A| = 40 = 2hk - \frac{h(3h - 1)}{2} - h |A \cap (- A)|+ 1 = \min \biggl(p, 2hk - \frac{h(3h - 1)}{2} + 1\biggr).\]
	In this example also, we have
	\[|3_{\pm}^\wedge A| = 40 < \min (p, 2hk - h^2 + 1).\]
\end{example}

\subsection{Restricted $[0, h]$-fold signed sumset $[0, h]_{\pm}^{\wedge} A$ in groups and fields}
In this subsection, we establish some results for the restricted $[0, h]$-fold signed sumset $[0, h]_{\pm}^{\wedge} A$ in groups and fields which gives the lower bound on the size of this sumset. The proof will require the following two results. 

\begin{theorem}[{\cite[Du and Pan]{du-pan2024}}] \label{du-pan-res}
	Let $A$ be a nonempty subset of a finite additive abelian group $G$ with $|G| > 1$. Let $h$ be a positive integer with $h \leq |A|$. Then
	\[|h^\wedge A| \ge \min (p(G), \, h|A| - h^2 + 1).\]
\end{theorem}

The following lemma will be useful for the proofs.

\begin{lemma}\label{rss-field-lem:1}
	Let $h$ be a positive integers. Let $A$ be a nonempty finite subset of an additive abelian group $G$ with $h \leq |A|$. Then
	\begin{equation}\label{rss-field-lem:1-eq1}
		[0, h]_{\pm}^\wedge A \supseteq [0, h]^\wedge (A \cup (- A) \cup \{0\}).
	\end{equation}	
	In particular, if $A \cap (- A) = \varnothing$, then
	\begin{equation}\label{rss-field-lem:1-eq2}
		[0, h]_{\pm}^\wedge A = [0, h]^\wedge (A \cup (- A) \cup \{0\}).
	\end{equation}
\end{lemma}

\begin{proof} If $s$ is a negative integer, then we define $s^\wedge (A) = \varnothing$ and $s_{\pm}^\wedge A = \varnothing$. For a nonnegative integer $s$, it is proved in \cite[p. 45]{bajnok2018}) that
	\[s^\wedge (A \cup (- A)) \subseteq \bigcup_{t = 0}^{\infty} (s - 2t)_{\pm}^\wedge A.\]
	Therefore, if $h$ is a positive integer, then for each $i \in [0, h]$, we can write 
	\[(h - i)^\wedge (A \cup (- A)) \bigcup_{t=0}^{\infty}\subseteq (h - i - 2t)_{\pm}^\wedge A.\]
	Hence for each $i \in [1, h]$, we have
	\begin{equation}\label{rss-field-lem:1-eq3}
		i^\wedge (A \cup (- A)) \subseteq [0, h]_{\pm}^\wedge A.	
	\end{equation}
	Also, it is to verify that for each $i \in [1, h]$, 
	\begin{equation}\label{rss-field-lem:1-eq4}
		i^\wedge (A \cup (- A) \cup \{0\}) \subseteq (i - 1)^\wedge (A \cup (- A)) \cup i^\wedge (A \cup (- A)).
	\end{equation}
	Therefore, it follows from \eqref{rss-field-lem:1-eq3} and \eqref{rss-field-lem:1-eq4} that
	\begin{equation}\label{rss-field-lem:1-eq5}
		[0, h]^\wedge (A \cup (- A) \cup \{0\}) \subseteq [0, h]_{\pm}^\wedge A.
	\end{equation}
	This proves \eqref{rss-field-lem:1-eq1}. 
	
	Now if $A \cap (- A) = \varnothing$, then it is easy to verify that
	\[i_{\pm}^\wedge A \subseteq i^\wedge (A \cup (- A) \cup \{0\})\]
	for each $i \in [0, h]$. Hence
	\begin{equation}\label{rss-field-lem:1-eq6}
		[0, h]_{\pm}^\wedge A \subseteq [0, h]^\wedge (A \cup (- A) \cup \{0\}).
	\end{equation}
	From \eqref{rss-field-lem:1-eq5} and \eqref{rss-field-lem:1-eq6}, we get  
	\[[0, h]_{\pm}^\wedge A = [0, h]^\wedge (A \cup (- A) \cup \{0\}).\]
	This proves \eqref{rss-field-lem:1-eq2}.
\end{proof}

Using these results, we can easily prove the following theorem in a finite additive abelian group $G$.

\begin{theorem}\label{rss-field-thm:5}
	Let $G$ be a finite additive abelian group $|G| > 1$. Let $h$ and $m$ be a positive integers with $2 \leq h \leq m$. Let $A$ be a nonempty finite subset of $G$ with $|A| = m$. Then
	\begin{equation}\label{rss-field-thm:5-eq1}
		[0, h]_{\pm}^\wedge A \geq
		\begin{cases}
			\min (p(G), 2hm - h^2 - h|A \cap (- A)| + 1), & ~\text{if}~ 0 \in A;\\
			\min (p(G), 2hm - h^2 - h|A \cap (- A)| + h + 1), & ~\text{if}~ 0 \not\in A.  
		\end{cases}
	\end{equation}
\end{theorem}	

\begin{proof}
	It follows from Lemma \ref{rss-field-lem:1} that
	\[|[0, h]_{\pm}^\wedge A| \geq |h^\wedge (A \cup (- A) \cup \{0\})|.\]
	Now an application of Theorem \ref{du-pan-res} yields
	\begin{equation}\label{rss-field-thm:5-eq2}
		[0, h]_{\pm}^\wedge A \geq \min (p(G), h|A \cup (- A) \cup \{0\}| - h^2 + 1).
	\end{equation}
	Since $|A \cup (-A)| = 2m - |A \cap (-A)|$, the inequalities in \eqref{rss-field-thm:5-eq1} follows easily from \eqref{rss-field-thm:5-eq2}.
\end{proof}

For a positive integer $m$ and a set of nonnegative integer $H$, we define
\[{\rho}_{{\pm}_\mathrm{asym}}^{\wedge}(G, m, H) = \min \{|H_{\pm}^\wedge A| : A \in \mathrm{Asym}(G, m)\},\]
\[{\rho}_{{\pm}_\mathrm{sym}}^{\wedge}(G, m, H) = \min \{|H_{\pm}^\wedge A| : A \in \mathrm{Sym}(G, m)\}\]
and
\[{\rho}_{{\pm}_\mathrm{nsym}}^{\wedge}(G, m, H) = \min \{|H_{\pm}^\wedge A| : A \in \mathrm{Nsym}(G, m)\}.\]

The following corollary follows immediately from the above theorem.
\begin{corollary}\label{rss-field-thm:4}
	Let $G$ is a finite additive abelian group with $|G| > 1$. Let $m$ and $h$ be integers such that $2 \leq h \leq m$. Then
	\begin{align*}
		{\rho}_{{\pm}_\mathrm{asym}}^{\wedge}(G, m, [0, h]) &\geq \min (p(G), 2hm - h^2 + h + 1), \\
		{\rho}_{{\pm}_\mathrm{sym}}^{\wedge}(G, m, [0, h]) &\geq \min (p(G), hm - h^2 + 1), \\
		{\rho}_{{\pm}_\mathrm{nsym}}^{\wedge}(G, m, [0, h]) &\geq \min (p(G), hm - h^2 + h + 1).
	\end{align*}
\end{corollary}

Next we prove some results for $[0, h]_{\pm}^\wedge A$ in an arbitrary field $\mathbb{F}$.

\begin{theorem}\label{rss-field-thm:6}
	Let $\mathbb{F}$ be a field. Let $h$ and $m$ be a positive integers with $h \leq m$. Let $A$ be a nonempty finite subset of $\mathbb{F}$ with $|A| = m$. Then
	\begin{equation}\label{rss-field-thm:6-eq1}
		[0, h]_{\pm}^\wedge A \geq
		\begin{cases}
			\min (p(\mathbb{F}), 2hm - h^2 - h|A \cap (- A)| + 1), & ~\text{if}~ 0 \in A;\\
			\min (p(\mathbb{F}), 2hm - h^2 - h|A \cap (- A)| + h + 1), & ~\text{if}~ 0 \not\in A.  
		\end{cases}
	\end{equation}
\end{theorem}	

\begin{proof}
	It follows from Lemma \ref{rss-field-lem:1} that
	\[|[0, h]_{\pm}^\wedge A| \geq |h^\wedge (A \cup (- A) \cup \{0\})|.\]
	Now an application of Theorem \ref{alon-nath-ruzsa} yields
	\begin{equation}\label{rss-field-thm:6-eq2}
		[0, h]_{\pm}^\wedge A \geq \min (p(\mathbb{F}), h|A \cup (- A) \cup \{0\}| - h^2 + 1).
	\end{equation}
	Since $|A \cup (-A)| = 2m - |A \cap (-A)|$, the inequalities in \eqref{rss-field-thm:6-eq1} follows easily from \eqref{rss-field-thm:6-eq2}.
\end{proof}

The following corollaries immediately follows from Theorem \ref{rss-field-thm:6}.

\begin{corollary}
	Let $\mathbb{F}$ be a field. Let $h$ and $m$ be positive integers with $h \leq m$. Then
	\begin{align*}
		{\rho}_{{\pm}_\mathrm{asym}}^{\wedge}(\mathbb{F}, m, [0, h]) &\geq \min (p(\mathbb{F}), 2hm - h^2 + h + 1), \\
		{\rho}_{{\pm}_\mathrm{sym}}^{\wedge}(\mathbb{F}, m, [0, h]) &\geq \min (p(\mathbb{F}), hm - h^2 + 1), \\
		{\rho}_{{\pm}_\mathrm{nsym}}^{\wedge}(\mathbb{F}, m, [0, h]) &\geq \min (p(\mathbb{F}), hm - h^2 + h + 1).
	\end{align*}
	
\end{corollary}

\begin{corollary}
	Let $h$ and $m$ be a positive integers with $h \leq m$. Let $A$ be a nonempty finite subset of $\mathbb{Z}$ with $|A| = m$. Then
	\begin{equation*}\label{cor-eq:1}
		[0, h]_{\pm}^\wedge A \geq
		\begin{cases}
			2hm - h^2 - h|A \cap (- A)| + 1, & ~\text{if}~ 0 \in A;\\
			2hm - h^2 - h|A \cap (- A)| + h + 1, & ~\text{if}~ 0 \not\in A.  
		\end{cases}
	\end{equation*}
\end{corollary}

	\end{document}